\documentclass[12pt]{amsart}

\usepackage[utf8]{inputenc}

\usepackage{amsfonts, amsmath, amssymb, amsthm, bbm, color, enumerate, graphicx, mathtools, tikz, hyperref, relsize}
\usepackage[numbers,sort]{natbib}
\usepackage{comment}

\usepackage[margin=30mm]{geometry}

\allowdisplaybreaks

\newtheorem{theorem}{Theorem}[section]
\newtheorem{corollary}[theorem]{Corollary}
\newtheorem{lemma}[theorem]{Lemma}
\newtheorem{proposition}[theorem]{Proposition}

\theoremstyle{definition}

\newtheorem{assumption}{Assumption}

\let\plainqed\qedsymbol
\newcommand{\claimqed}{$\lrcorner$}

\hypersetup{
 colorlinks=true,
 linkcolor=blue,          
 citecolor=red,       
 filecolor=red,   
 urlcolor=red, 
 pdftitle={},
 pdfauthor={},
 pdfsubject={},
 pdfkeywords={},
 linktocpage=true
}

\usepackage{fancyhdr}
 
\usepackage{natbib}

\usepackage{amssymb}
\usepackage{amsmath}
\usepackage{amsthm}

\newcommand{\om}{\omega}

\newcommand{\Om}{\mathnormal{\Omega}}

\newcommand{\NN}{{\mathbb N}}

\newcommand{\RR}{{\mathbb R}}
\newcommand{\R}{{\mathbb R}}

\newcommand{\EE}{{\mathbb E}}

\newcommand{\PP}{{\mathbb P}}
\newcommand{\QQ}{{\mathbb Q}}

\newcommand{\cla}{{\mathcal A}}

\newcommand{\clb}{{\mathcal B}}
\newcommand{\calC}{{\mathcal C}}
\newcommand{\clc}{{\mathcal C}}

\newcommand{\clf}{{\mathcal F}}

\newcommand{\cll}{{\mathcal L}}

\newcommand{\clm}{{\mathcal M}}

\newcommand{\clp}{{\mathcal P}}

\newcommand{\calR}{{\mathcal R}}
\newcommand{\clr}{{\mathcal R}}

\newcommand{\cls}{{\mathcal S}}

\newcommand{\clu}{{\mathcal U}}

\newcommand{\clw}{{\mathcal W}}
\newcommand{\calX}{{\mathcal X}}

\newcommand{\cly}{{\mathcal Y}}

\newcommand{\BM}{\mbox{BM}}
\newcommand{\lip}{\mbox{\tiny{lip}}}
\newcommand{\veps}{\varepsilon}
\newcommand{\be}{\begin{equation}}
\newcommand{\ee}{\end{equation}}

\newcommand{\limsupp}{\overline{\lim}^{\bar\PP}}
\newcommand{\liminff}{\underline{\lim}^{\bar\PP}}
\numberwithin{equation}{section}
\title[Small Noise Filtering]{Some Large Deviations Asymptotics in Small Noise Filtering Problems}
\author{Anugu Sumith Reddy}  
\address{Anugu Sumith Reddy, International Centre for
	Theoretical Sciences-Tata Institute of Fundamental Research,
	Bangalore, India}\email{anugu.reddy@icts.res.in}
\author{Amarjit Budhiraja}
\address{Amarjit Budhiraja,
Department of Statistics and Operations Research, 304 Hanes Hall, University of North Carolina, Chapel Hill, NC 27599}\email{amarjit@unc.edu}
\author{Amit Apte} 
\address{Amit Apte, International Centre for
	Theoretical Sciences-Tata Institute of Fundamental Research,
	Bangalore, India}\email{apte@icts.res.in}
\keywords{Laplace asymptotics, large deviation principle, nonlinear filtering, small observation and signal noise, minimum energy estimate,  $4$DVAR}
\subjclass[2010]{60F10, 60G35, 93E11}

\begin{document}
\maketitle
\begin{abstract}
We consider nonlinear filters for diffusion processes when the observation and signal noises are small and of the same order. As the noise intensities approach zero, the nonlinear filter can be approximated by a certain variational problem that is closely related to Mortensen's optimization problem(1968). This approximation result can be made precise through a certain Laplace asymptotic formula. In this work we study probabilities of deviations of true filtering estimates from that obtained by solving the variational problem. Our main result gives a large deviation principle for Laplace functionals whose typical asymptotic behavior is described by Mortensen-type variational problems. Proofs rely on stochastic control representations for positive functionals of Brownian motions and Laplace asymptotics of the Kallianpur-Striebel formula.
\end{abstract}

\section{Introduction}
In this work we study certain large deviation asymptotics for nonlinear filtering problems with small signal and observation noise. As the noise in the signal and observation processes vanish, the filtering problem can formally be replaced by a variational problem and one may approximate the filtering estimates (namely suitable conditional probabilities or expectations) by solutions of certain deterministic optimization problems. However due to randomness there will be occasional large deviations of the true nonlinear filter estimates from the variational problem solutions. The main goal of this work is to investigate the probabilities of such deviations by establishing a suitable large deviation principle. Large deviations and related asymptotic problems in the context of small noise nonlinear filtering have been investigated, under different settings, in many  works \cite{hijab, Heunis1987nonlinear, baras1988dynamic, james1988nonlinear, parzei, berdjeouk, xiong, marxio2, marxio1, doss1991nouveau, rabeherimanana1992petites, arous1996flow}. We summarize the main results of these works and their relation to the asymptotic questions considered in the current work at the end of this section.

In order to describe our results precisely, we begin by introducing the filtering model that we study.
We consider a signal process $X^{\veps}$ given as the solution of the $d$-dimensional stochastic differential equation (SDE)
\begin{equation}\label{eq:sdex}
	dX^{\veps}(t) = b(X^{\veps}(t)) dt + \veps\sigma(X^{\veps}(t)) dW(t), \; X^{\veps}(0)= x_0, \; 0 \le t \le T,
\end{equation}
and an $m$-dimensional observation process $Y^{\veps}$ governed by the equation
\begin{equation}\label{eq:sdey}
	Y^{\veps}(t) = \int_0^t h(X^{\veps}(s)) ds + \veps B(t), \; 0 \le t \le T
\end{equation}
on some probability space $(\bar{\Om} , \bar{\clf}, \bar{\PP})$.
Here $\veps \in (0,\infty)$ is a small parameter, $T\in (0,\infty)$ is some given finite time horizon, $W$ and $B$ are mutually independent standard Brownian motions in
$\RR^k$ and $\RR^m$ respectively, $x_0\in \RR^d$ is known deterministic initial condition of the signal, and the functions $b,\sigma$ and $h$ are required to satisfy the following condition.
\begin{assumption} \label{assu:main} The following hold.
	\begin{enumerate}[(a)]
		\item The functions $b, \sigma, h$ from $\R^d \to \R^d$, $\R^d \to \R^{d\times k}$, $\R^d \to \R^m$ are Lipschitz: For some $c_{\lip}\in (0,\infty)$
		$$\|b(x)-b(y)\| + \|\sigma(x)-\sigma(y)\| + \|h(x)-h(y)\| \le c_{\lip} \|x-y\| \mbox{ for all } x,y\in \RR^d.$$
		\item The function $\sigma$ is bounded: For some $c_{\sigma}\in (0,\infty)$
		$$\sup_{x\in \R^d} \|\sigma(x)\| \le c_{\sigma}.$$
		\item The function $h$ is twice continuously differentiable with bounded first and second derivatives.
	\end{enumerate}
 \end{assumption}
 Note that under Assumption \ref{assu:main} there is a unique pathwise solution of \eqref{eq:sdex} and the solution is a stochastic process with sample paths in
 $\clc_d$ (the space of continuous functions from $[0,T]$ to $\RR^d$ equipped with the uniform metric).
 
 The filtering problem is concerned with the computation of the conditional expectations of the form
\begin{equation}
	\bar\EE\left[\phi(X^{\veps}) \mid \cly^{\veps}_T\right]
\end{equation}
where $\cly^{\veps}_T \doteq \sigma\{Y^{\veps}(s): 0 \le s \le T\}$ and 
$\phi: \clc_d \to \RR$ is a suitable map. The stochastic process with values in the space of probability measures on $\clc_d$, given by
$$\bar\PP\left[X^{\veps} \in \cdot\mid \cly^{\veps}_T\right]$$
is usually referred to as the nonlinear filter.

In this work we are interested in the study of the asymptotic behavior of the nonlinear filter as $\veps\to 0$.
 Denote by $\xi^* \in \clc_d$ the unique solution of
 \begin{equation}\label{eq:intro1}
 	d\xi^*(t) = b(\xi^*(t)) dt, \;\; \xi^*(0)= x_0.
 \end{equation}
 It can be shown that, under additional conditions (see discussion in Section \ref{sec:canpath}), that, as $\veps \to 0$,
 \begin{equation}\label{eq:todiracxistar}
 	\bar\PP\left[X^{\veps} \in \cdot\mid \cly^{\veps}_T\right] \to \delta_{\xi^*}, \mbox{ in probability, under } \bar\PP,
 \end{equation}
 weakly.
 In particular for Borel subsets $A$ of $\clc_d$ whose closure does not contain $\xi^*$ one will have $\bar\PP\left[X^{\veps} \in A\mid \cly^{\veps}_T\right]\to 0$ in probability as $\veps \to 0$.
 It is of interest to study the rate of decay of conditional probabilities of such non-typical state trajectories.
 As a special case of the results of the current paper (see Corollary \ref{cor:keycgcecor}) it will follow that
for every real continuous and bounded function $\phi$  on $\clc_d$,  denoting
\begin{equation}\label{Uveps}
U^{\veps}[\phi]\doteq \bar \EE\left[ e^{-\frac{1}{\veps^2} \phi(X^{\veps})} \mid \cly^{\veps}_T\right],
\end{equation}
\begin{equation}\label{eq:439}
	-\veps^2 \log U^{\veps}[\phi] \stackrel{ \bar\PP}{\longrightarrow}
	\inf_{\eta \in \clc_d} \left[\phi(\eta)+ \frac{1}{2} \int_0^T \|h(\eta(s)) - h(\xi^*(s))\|^2  + J(\eta)\right],
\end{equation}
where $\stackrel{ \bar\PP}{\longrightarrow}$ denotes convergence in probability under $\bar\PP$,
and $J$ is the rate function on $\clc_d$ associated with the large deviation principle for $\{X^{\veps}\}_{\veps>0}$ (see Section \ref{sec:canpath}).
From this convergence it follows using standard arguments (see e.g. \cite[Theorem 1.8]{budhiraja2019analysis}),
that, for all Borel subsets $A$ of $\clc_d$)
\begin{equation}\label{eq:450}
	\begin{aligned}
 	\liminff_{\veps\to 0} \veps^2 \log  \bar\PP\left[X^{\veps} \in A \mid \cly^{\veps}_T\right] &\ge  -\inf_{\eta \in A^o} \left[ \frac{1}{2} \int_0^T \|h(\eta(s)) - h(\xi^*(s))\|^2  + J(\eta)\right]
 	\\
 	 \limsupp_{\veps\to 0} \veps^2 \log\bar\PP\left[X^{\veps} \in A \mid \cly^{\veps}_T\right]& \le -\inf_{\eta \in \bar A} \left[ \frac{1}{2} \int_0^T \|h(\eta(s)) - h(\xi^*(s))\|^2
   + J(\eta)\right], 
\end{aligned}
\end{equation}
where for real random variables $Z^{\veps}$ and a constant $\alpha \in \RR$ we say
$\limsupp_{\veps\to 0} Z^{\veps} \le \alpha$ [ resp. $\liminff_{\veps\to 0} Z^{\veps} \ge \alpha$]
if $(Z^{\veps}-\alpha)^+$ [resp. $(\alpha -Z^{\veps})^+$]  converges to $0$ in $\bar\PP$-probability, and for a set $A$, $A^o$ and $\bar A$ denote its interior and closure respectively..

Thus the convergence in \eqref{eq:439} gives information on rate of decays of  conditional probabilities of  non-typical state trajectories. Formally, denoting the infimum in the above display as $S(\xi^*, A)$, we can write approximations for conditional probabilities:
\begin{equation}\label{eq:approxeq}\bar\PP\left[X^{\veps} \in A \mid \cly^{\veps}_T\right] \approx \exp\left\{-\frac{1}{\veps^2} S(\xi^*,A)\right\}.\end{equation}
However,  due to stochastic fluctuations, one may find that for some  `rogue' observation trajectories the  conditional probabilities on the left side of \eqref{eq:approxeq} are quite different from  the deterministic approximation on the right side of  \eqref{eq:approxeq}.
In order to quantify the probabilities of observing such rogue observation trajectories that cause deviations from the bounds in \eqref{eq:450}, a natural approach is to study a large deviation principle for $\RR$ valued random variables $\{-\veps^2 \log U^{\veps}[\phi]\}$ whose typical (law of large numbers) behavior is described by the right side of \eqref{eq:439}.
Establishing such a large deviation principle is the goal of this work. 
Such a result gives information on decay rates of probabilities of the form
$$
\begin{aligned} &\bar \PP\Big\{\Big|\veps^2 \log \bar \PP\left[ X^{\veps} \in A \mid \cly^{\veps}_T\right]\\
	& \quad\quad\quad  +\inf_{\eta \in A} \left[\frac{1}{2} \int_0^T \|h(\eta(s)) - h(\xi^*(s))\|^2  + J(\eta)\right]\Big| >\delta\Big\}\end{aligned}$$
for suitable sets $A \in \clb(\clc_d)$ and $\delta>0$.
Our main result is
Theorem \ref{thm:main} which gives a large deviation principle for $\{-\veps^2 \log U^{\veps}[\phi]\}$, for every continuous and bounded function $\phi$ on $\clc_d$ with a rate funcion defined by the variational formula in \eqref{eq:clszchA}-\eqref{eq:clszch}.\\

\noindent {\bf Notation.} The following notation and definitions will be used.  For $p \in \NN$ the Euclidean norm in $\RR^p$ will be denoted as $\|.\|$ and the corresponding inner product will be written as
$\langle\cdot , \cdot \rangle$. The space of finite positive measures (resp. probability measures) on a Polish space $S$ will be denoted by 
$\clm(S)$ (resp. $\clp(S)$). The space of bounded measurable (resp. continuous and bounded) functions from $S \to \RR$ will be denoted by
$\BM(S)$ and $C_b(S)$ respectively.   For $\phi \in \BM(S)$, $\|\phi\|_{\infty} \doteq \sup_{x\in S} |\phi(x)|$.
For $\phi \in \BM(S)$ and $\mu \in \clm(S)$, $\mu[\phi]\doteq \int \phi\, d\mu$. Borel $\sigma$-field on a Polish space $S$ will be denoted as $\clb(S)$.
For $p \in \NN$ and $T\in (0,\infty)$, $\clc_{p,T}$ will denote the space of continuous functions from $[0,T]$ to $\RR^p$ which is equipped with the supremum norm, defined as $\|f\|_{*,T} \doteq \sup_{0\le t \le T}\|f(t)\|$, $f \in \clc_{p,T}$. Since $T\in (0,\infty)$ will be fixed in most of this work, frequently the subscript $T$ in
$\clc_{p,T}$ and $\|f\|_{*,T}$ will be dropped. We denote by $\cll^2_p \doteq L^2([0,T]:\RR^p)$ the Hilbert space of square-integrable functions from $[0,T]$ to $\RR^p$.
By convention, the infimum over an empty set will be taken to be $\infty$. For random variables $X_n$, $X$ with values in some Polish space $S$, convergence in distribution of $X_n$ to $X$ will be denoted as $X_n \Rightarrow X$.
A function $I$ from a Polish space $S$ to $[0,\infty]$ is called a rate function if it has compact sub-level sets, namely the set $\{x\in S: I(x) \le m\}$ is a compact set of $S$ for every $m \in (0,\infty)$. 
Given a function $a: (0, \infty) \to (0, \infty)$ such that $a(\veps) \to \infty$ as $\veps \to 0$,  and a rate function $I$ on a Polish space $S$, a collection $\{U^{\veps}\}_{\veps>0}$ of $S$ valued random variables is  said to satisfy a large deviation principle (LDP) with rate function $I$ and speed $a(\veps)$ if for every $\phi \in C_b(S)$
$$\lim_{\veps \to 0}-a(\veps)^{-1} \log \EE e^{-a(\veps) \phi(U^{\veps}) } = \inf_{x\in S} [I(x) + \phi(x)].$$

\noindent {\bf Relation with existing  body of work.} 
 Denote by $\clc^1_m$ the collection of absolutely continuous functions $y \in \clc_m$ that satisfy $\int_0^T \|\dot y(s)\|^2 ds <\infty$.
 For $y \in \clc^1_m$ define $I_{y}:\clc_d \to [0,\infty]$ as
\begin{align}\label{eq:izeta}
I_y(\eta)=
\frac{1}{2}\int_0^T \|h(\eta(s))-\dot y(s)\|^2 ds + J(\eta)
\end{align}
 where $J$ is the rate function of $\{X^{\veps}\}$ defined in \eqref{eq:jdefn}. The functional $I_{y}$ was introduced in Mortensen\cite{Mortensen1968maximum}  as the objective function in an optimization problem whose minima describes the most probable trajectory given the data in a nonlinear filtering problem in an appropriate asymptotic sense. This functional is also used in implementing the popular $4$DVAR data assimilation algorithm (cf. \cite[Section 3.2]{carrassi2018data}, \cite[Chapter 16]{fletcher2017data}). Connection of the optimization problem associated with the objective function in \eqref{eq:izeta}  with the asymptotics of classical small noise filtering problem has been studied by several authors \cite{hijab,hijab1980minimum,james1988nonlinear}.
We now describe this connection.

In Section \ref{sec:canpath} we will introduce a continuous map $\hat \Lambda^{\veps}_T: \clc_m \to \clp(\clc_d)$ such that $\hat \Lambda^{\veps}_T(Y^{\veps}) = \bar \PP(X^{\veps} \in \cdot \mid \cly^{\veps}_T)$ a.s.
In \cite{hijab}, Hijab established, under conditions (that include boundedness and smoothness of various coefficients functions), a large deviation principle for the collection of probability measures  (on $\clc_d$) $\{\hat \Lambda^{\veps}_T(y)\}_{\veps>0}$ (with speed $\veps^{-2}$), for a fixed $y$ in $\clc^1_m$, with  rate function  $\hat I_{y}: \clc_d \to [0,\infty]$ given by
\begin{equation}
	\hat I_{y}(\eta) = I_{y}(\eta) - \inf_{\hat{\eta} \in \clc_d}\{ I_{y}(\hat{\eta})\}.
\end{equation}
In a related direction, Hijab's Ph.D. dissertation \cite{hijab1980minimum}, studied asymptotics of the unnormalized conditional density and established, under conditions, an asymptotic formula of the form
$$q^{\veps}(x,t) = \exp\left\{- \frac{1}{\veps^2} (W(x,t)+o(1))\right\},$$
where $q^{\veps}(x,t)$ denotes the solution of the Zakai equation associated with the nonlinear filter (cf. \cite{kallianpur2013stochastic}). The deterministic function $W(x,t)$ coincides with Mortensen's (deterministic) minimum energy estimate \cite{Mortensen1968maximum} which is given as solution of a certain minimization problem related to the objective function $I_{y}(\eta)$. 
Results of Hijab were extended to random initial conditions in \cite{james1988nonlinear}, once again assuming boundedness and smoothness of coefficients. In related work, the problem of constructing observers for dynamical systems as limits of stochastic nonlinear filters is studied in \cite{baras1988dynamic}.  Heunis\cite{Heunis1987nonlinear} studies a somewhat different asymptotic problem for small noise nonlinear filters. Specifically, it is shown in \cite{Heunis1987nonlinear}, that for every $\phi \in C_b(\clc_d)$, $w\in \clc_m$, and for any $\eta \in \clc_d$ for which the map defined in \eqref{eq:jhh} has a {\em unique} minimizer (at say $\eta^*$),
$$\hat \Lambda^{\veps}_T \left(\int_0^{\cdot} h(\eta(s)) ds + \veps w\right)[\phi] \to \phi(\eta^*), \; \mbox{ as }  \veps \to 0.$$
This result and its connection to our work are further discussed in Section \ref{sec:canpath}. In particular the statement in \eqref{eq:todiracxistar} follows readily on using similar ideas as in \cite{Heunis1987nonlinear}.
The work of Pardoux and Zeitouni\cite{parzei} considers a one dimensional nonlinear filtering problem where the observation noise is small while the signal noise is $O(1)$ (specifically, the term $\veps \sigma(X^{\veps}(t))$ in \eqref{eq:sdex} is replaced by $1$). In this case the conditional distribution of $X(T)$ given $\cly^{\veps}_T$ converges a.s. to a Dirac measure $\delta_{X(T)}$ as $\veps \to 0$. The paper \cite{parzei}  proves a quenched 
LDP for this conditional distribution (regarded as a collection of probability measures on $\clc_d$ parametrized by $X(T)(\omega)$) in $\clc_d$. In a somewhat different direction, in a sequence of papers \cite{xiong, marxio1, marxio2}, the authors have studied asymptotics of the filtering problem under a small signal to noise ratio limit, under various types of model settings. In this case the nonlinear filter converges to the unconditional law of the signal and the authors establish large deviation principles characterizing probabilities of deviation of the filter from the above deterministic law. An analogous result in a correlated signal-observation noise case was studied in \cite{berdjeouk}. Finally, yet another type of large deviation problem in the context of nonlinear filtering (with correlated signal-observation noise) when the observation noise is $O(1)$ and the signal noise and drift are suitably small has been considered in a series of papers \cite{doss1991nouveau, rabeherimanana1992petites, arous1996flow}.

The closest connections of the current work are with \cite{hijab} and \cite{Heunis1987nonlinear}. Specifically, the asymptotic statements in \eqref{eq:439} and  \eqref{eq:450} which follow as a special case of our results (see Corollary \ref{cor:keycgcecor}) is analogous to results in \cite{hijab}, except that instead of  a fixed observation path we consider the actual observation process $Y^{\veps}$ (also we make substantially weaker assumption on coefficients than \cite{hijab}). However our main interest is in a large deviation principle for the convergence to the deterministic limit in \eqref{eq:439} , thus roughly speaking we are interested in quantifying the probability of deviations from the convergence statement in \cite{hijab} (when a fixed observation path is replaced with the observation process $Y^{\veps}$). This large deviation result, given in Theorem \ref{thm:main},
is the main contribution of our work.
\\

\noindent{\bf Proof idea.}
The proof of Theorem \ref{thm:main} is based on a variational representation for functionals of Brownian motions obtained in \cite{boue1998variational} (see also \cite{budhiraja2000variational}) using which the proof of the large deviation principle reduces to proving a key weak convergence result given in Lemma \ref{lem:keycgce}. Proof of Lemma \ref{lem:keycgce} is the technical heart of this work. Important use is made of some key estimates obtained in \cite{Heunis1987nonlinear} (see in particular Proposition \ref{prop:heunis}). One of the key steps is to argue that terms of order $\veps^{-1}$ can be ignored in the exponent when studying Laplace asymptotics for the quantity on the left side of  \eqref{eq:repboudup}. This relies on several careful large deviation exponential estimates which are developed in Section \ref{sec:pflem4.1}. Once Lemma \ref{lem:keycgce} is available the proof of the large deviation principle in Theorem \ref{thm:main} follows readily using the now well developed weak convergence approach for the study of large deviation problems (cf. \cite{budhiraja2019analysis}).
\\

\noindent{\bf Organization.} It will be convenient to formulate the filtering problem on canonical path spaces and also to represent the nonlinear filter through a map given on the path space of the observation process. This formulation and our main result (Theorem \ref{thm:main})  are given in Section \ref{sec:canpath}.  The key idea in the proof of the LDP is a variational representation from \cite{boue1998variational}. A somewhat simplified version of this representation (cf. \cite{budhiraja2019analysis}) that is used in this work is presented in Section \ref{sec:varrep}. Section \ref{sec:keylem} presents a key lemma (Lemma \ref{lem:keycgce}) that is needed for implementing the weak convergence method for proving the large deviation result in Theorem \ref{thm:main}. Section \ref{sec:pflem4.1} is devoted to the proof of Lemma \ref{lem:keycgce}. Using this lemma, the proof of Theorem \ref{thm:main} is completed in Section \ref{sec:pfmainth}. 
\\

\section{Setting and Main Result}
\label{sec:canpath}

Recall that $X^{\veps}$ has sample paths in $\clc_d$.  Similarly, the processes $Y^{\veps}, W, B$ have sample paths in $\clc_m, \clc_k, \clc_m$ respectively.
It will be convenient to formulate the filtering problem on suitable path spaces.
Denote, for $p\in \NN$, the standard Wiener measure on $(\clc_p, \clb(\clc_p))$ as $\clw_p$ and the Wiener measure with variance parameter $\veps^2$ as
$\clw_p^{\veps}$.
Denote the canonical coordinate process on $(\clc_k, \clb(\clc_k))$ as $\{\gamma(t): 0 \le t \le T\}$
and consider the SDE on the probability space $(\clc_k, \clb(\clc_k), \clw_k)$,
$$dx^{\veps}(t) = b(x^{\veps}(t)) dt + \veps\sigma(x^{\veps}(t)) d\gamma(t), \; x^{\veps}(0)= x_0, \; 0 \le t \le T.$$
From Assumption \ref{assu:main}, the above SDE has a unique strong solution with sample paths in $\clc_d$.

Consider the map $\clc_k \to \Om_x\doteq \clc_d\times \clc_k$ defined as
$\om \mapsto (x^{\veps}(\om), \gamma(\om))$ and let
$$\mu^{\veps} \doteq \clw_k \circ (x^{\veps}, \gamma)^{-1}.$$
Next, let $\Om_y \doteq \clc_m$ and consider the probability space
$$(\Om, \clf, \QQ^{\veps}) \doteq (\Om_x\times \Om_y, \clb(\Om_x)\otimes \clb(\Om_y), \mu^{\veps}\otimes \clw^{\veps}_m).$$
Abusing notation, denote the coordinate maps on the above probability space as $\xi, \gamma,  \zeta$, namely
$$\xi(\om)= \om_1, \;\; \gamma (\om) = \om_2, \;  \; \zeta(\om) = \om_3 \mbox{ for } \om = ((\om_1, \om_2), \om_3) \in \Om_x\times \Om_y.$$
We will frequently write $\xi(\om)(s)$ as $\xi(s)$ for $(\om, s)\in \Om \times [0,T]$. Similar notational shorthand will be followed for other coordinate maps.

Note that, under $\QQ^{\veps}$, $\xi(0) = x_0$, $\gamma$ and $\veps^{-1}\zeta$ are independent standard Brownian motions in $\RR^k$ and $\RR^m$ respectively and
\begin{equation}\label{eq:sigcanon}
	\xi(t) = x_0 + \int_0^t b(\xi(s)) ds + \veps\int_0^t \sigma(\xi(s)) d \gamma(s), \; 0 \le t \le T.
\end{equation}
Define, for $\QQ^{\veps}$ a.e. $\om = ((\om_1, \om_2), \om_3)$, for $t \in [0,T]$,
$$L^{\veps}_t(\om) \doteq \exp \left\{ \frac{1}{\veps^2} \int_0^t \langle h(\xi(s)), d\zeta(s)\rangle - \frac{1}{2\veps^2} \int_0^t \|h(\xi(s))\|^2 ds \right\}.$$
Note that, since under $\QQ^{\veps}$, $\veps^{-1}\zeta$ is a standard Brownian martingale with respect to the filtration 
$\clf_t^0 \doteq \sigma\{\gamma(s), \xi(s), \zeta(s): 0 \le s \le t\}$, the first integral in the exponent is well-defined as an It\^{o} integral.
From the independence of $\xi$ and $\zeta$ under $\QQ^{\veps}$ and Assumption \ref{assu:main} it follows that
$L^{\veps}_t$ is a $\{\clf^0_t\}$-martingale under $\QQ^{\veps}$. Define a probability measure $\PP^{\veps}$ on $(\Om, \clf)$ as
$$\frac{d\PP^{\veps}}{d\QQ^{\veps}}(\om) \doteq L^{\veps}_T(\om), \; \QQ^{\veps} \mbox{ a.e. } \om .$$
Note that, by Girsanov's theorem, under $\PP^{\veps}$
\begin{equation}
	\beta(t) \doteq \frac{1}{\veps}\zeta(t) - \frac{1}{\veps} \int_0^t h(\xi(s)) ds , \; 0 \le t \le T
\end{equation}
is a standard $m$-dimensional Brownian motion which is independent of $(\xi, \gamma)$.
Rewriting the above equation as
$$\zeta(t) = \int_0^t h(\xi(s)) ds + \veps \beta(t), \; 0 \le t \le T,$$
we see that
\begin{equation}
	\bar \PP \circ (X^{\veps}, Y^{\veps})^{-1} = \PP^{\veps} \circ (\xi, \zeta)^{-1}.
\end{equation}
Next, for $\veps>0$, define $\Gamma^{\veps}_T: \clc_m \to \clm(\clc_d)$ as
\begin{equation}
	\Gamma^{\veps}_T(\om_3)[A] \doteq \int_{\Om_x} 1_A(\om_1) L^{\veps}_T((\om_1, \om_2), \om_3) d\mu^{\veps}(\om_1, \om_2), \; \om_3 \in \clc_m,\; A \in \clb(\clc_d).
\end{equation}
The maps  are well defined $\PP^{\veps}$-a.s. and using results of  \cite{clark, davis1,davis2}, one can obtain  versions of these maps (denoted as $\hat\Gamma^{\veps}_T$) which are continuous on $\clc_m$.
Also, define $\Lambda^{\veps}_T: \clc_m \to \clp(\clc_d)$ as
\begin{equation}
	\Lambda^{\veps}_T(\om)[A] \doteq \frac{\Gamma^{\veps}_T(\om)[A]}{\Gamma^{\veps}_T(\om)[\clc_d]}, \; \PP^{\veps}\mbox{-a.e. } \om \in \clc_m,\; A \in \clb(\clc_d).
\end{equation}
Once again, for each $\veps>0$, this map is well defined $\PP^{\veps}$-a.s.  and a continuous version of the map exists (which we denote as  $\hat \Lambda^{\veps}_T$) from  \cite{clark, davis1,davis2} .
Write, for $f \in \BM(\clc_d)$
$$\Gamma^{\veps}_T(f,\om)\doteq \int_{\clc_d} f(\tilde\om) \Gamma^{\veps}_T(\om)[d\tilde \om], \; \Lambda^{\veps}_T(f,\om) \doteq \int_{\clc_d} f(\tilde\om) \Lambda^{\veps}_T(\om)[d\tilde \om], \; \PP^{\veps}\mbox{-a.e. } \om \in \clc_m.$$
Then with $(X^{\veps}, Y^{\veps})$ as in \eqref{eq:sdex}-\eqref{eq:sdey}, for $\phi \in \BM(\clc_d)$
\begin{equation}\label{Lambdadefn}
	\bar \EE\left[\phi(X^{\veps}) \mid \cly^{\veps}_T\right] = \Lambda^{\veps}_T(\phi, Y^{\veps}) \mbox{ a.s. } \bar \PP.
\end{equation}
Also,
\begin{equation}
	 \EE_{\PP^{\veps}}\left[\phi(\xi) \mid \sigma\{\zeta(s): 0 \le s \le T\}\right] = \Lambda^{\veps}_T(\phi, \zeta) \mbox{ a.s. } \PP^{\veps},
\end{equation}
where $\EE_{\PP^{\veps}}$ denotes the expectation under the probability measure $\PP^{\veps}$, and
\begin{equation}\label{eq:samelaw}
	\bar \PP \circ (X^{\veps}, Y^{\veps}, W, B, \Lambda^{\veps}_T(\phi, Y^{\veps}))^{-1} = \PP^{\veps} \circ (\xi, \zeta, \gamma, \beta, \Lambda^{\veps}_T(\phi, \zeta))^{-1}.
\end{equation}
Let for $\xi_0 \in \clc_d$,
\begin{equation}\label{eq:jdefn}
	J(\xi_0) \doteq \inf_{\varphi \in \clu(\xi_0)}\left[ \frac{1}{2} \int_0^T \|\varphi(t)\|^2 dt\right]\end{equation}
where
$\clu(\xi_0)$ is the collection of all $\varphi$ in $\cll^2_k$ such that
\begin{equation}\label{eq:eqxio}
	\xi_0(t) = x_0 + \int_0^t b(\xi_0(s)) ds + \int_0^t \sigma(\xi_0(s)) \varphi(s) ds, \; t \in [0,T].
\end{equation}
Note that, by Assumption \ref{assu:main}, for every $\varphi \in \cll^2_k$ there is a unique solution of \eqref{eq:eqxio}.
By classical results of Freidlin and Wentzell (see e.g. \cite[Theorem 10.6]{budhiraja2019analysis})
the collection $\{X^{\veps}\}$ of $\clc_d$ valued random variables satisfies a LDP with rate function $J$ and speed $\veps^{-2}$, namely, for all
$F \in C_b(\clc_d)$
\begin{equation}\label{eq:ldpsign}
	\lim_{\veps \to 0}-\veps^2 \log \int_{\Om_x} \exp\left\{-\frac{1}{\veps^2} F(\hat\xi)\right\} d\mu^{\veps} = \inf_{\xi_0 \in \clc_d}\left[ F(\xi_0) + J(\xi_0)\right],
\end{equation}
where  we denote the first coordinate process on $\Om_x$ by $\hat\xi$, i.e. $\hat\xi(\om) = \om_1$ for $\om = (\om_1, \om_2) \in \Om_x = \clc_d\times \clc_k$.
In \cite{Heunis1987nonlinear} it is shown  that for every  $w \in \clc_m$, and a given $\eta \in \clc_d$, the probability measure
\begin{equation}
	\hat \Lambda_T^{\veps}\left(\int_0^{\cdot} h(\eta(s))ds + \veps w(\cdot) \right) \to \delta_{\eta^*}
\end{equation}
weakly,
{\em if}  the 
 map
\begin{equation}\label{eq:jhh}
	\tilde \eta \mapsto J(\tilde \eta) +\frac{1}{2} \int_0^T  \|h(\eta(s)) - h(\tilde \eta(s))\|^2 ds
	\end{equation}
attains its infimum over $\clc_d$  {\em uniquely} at $\eta^*$,
where recall that   $\hat \Lambda^{\veps}_T$ is the continuous version of $\Lambda^{\veps}_T$.
We remark that \cite{Heunis1987nonlinear} assumes in addition to \eqref{assu:main} that $h$ and $b$ are bounded, but an examination of the proof shows (see calculations in Section \ref{sec:pflem4.1}) that these conditons can be replaced by linear growth conditions that are implied by Assumption \ref{assu:main} .
 
 Recall the function $\xi^* \in \clc_d$ from \eqref{eq:intro1}.
 Then using similar ideas as in \cite{Heunis1987nonlinear}, under Assumption \ref{assu:main}, and assuming in addition that either $\sigma \sigma^\dagger$ is positive definite or $h$ is a one-to-one function, it follows that
 \begin{equation}
 	\Lambda_T^{\veps} \to \delta_{\xi^*}, \mbox{ in probability, under } \PP^{\veps},
 \end{equation}
 weakly, as $\veps \to 0$.
 This is a consequence of the fact that when  $\eta = \xi^*$ the map in \eqref{eq:jhh} achieves its minimum (which is $0$) uniquely at $\xi^*$.

 As a consequence of the results of the current paper (see Corollary \ref{cor:keycgcecor}) one can show the Laplace asymptotic formula in \eqref{eq:439}. Recall from the discussion in the Introduction that the convergence in \eqref{eq:439} gives information on asymptotics  of  conditional probabilities of  non-typical state trajectories. 
In order to quantify the decay rate of probabilities of observing rare observation trajectories that cause deviations from the deterministic variational quantity in \eqref{eq:439}, we will establish a large deviation principle for  $\{-\veps^2 \log U^{\veps}[\phi]\}$ defined in \eqref{Uveps}.

We now present the rate function associated with this LDP.

Define the map $H: \clc_d \times \clc_d \times \cll^2_m\to \RR_+$ as
\begin{equation}
	H(\eta, \tilde \eta, \psi) \doteq \frac{1}{2} \int_0^T \|h(\eta(s)) - h(\tilde \eta(s)) - \psi(s)\|^2 ds.
\end{equation}
Also, for $\varphi \in \cll^2_k$, let $\xi_0^{\varphi}$ be given as the unique solution of \eqref{eq:eqxio}.

We now introduce the rate function that will govern the large deviation asymptotics of $-\veps^2 \log U^{\veps}[\phi]$.

Fix $\phi \in C_b(\clc_d)$ and define $I^{\phi}: \RR \to [0,\infty]$ as
\begin{equation}\label{eq:clszchA}
	I^{\phi}(z) = \inf_{(\varphi, \psi) \in \cls(z)} \left[\frac{1}{2} \int_0^T \|\varphi(t)\|^2 dt + \frac{1}{2} \int_0^T \|\psi(t)\|^2 dt\right]
\end{equation}
where $\cls(z)$ is the collection of all $(\varphi, \psi)$ in
$\cll^2_k\times \cll^2_m$ such that
\begin{equation} \label{eq:clszch}
	\inf_{\eta \in \clc_d} \left[ H(\eta,\xi_0^{\varphi},  \psi) + \phi(\eta) + J(\eta)\right]
	- \inf_{\eta \in \clc_d} \left[ H(\eta,\xi_0^{\varphi}, \psi)  + J(\eta)\right] = z.
\end{equation}

The following is the main result of the work.
\begin{theorem}
	\label{thm:main}
	Suppose that Assumption \ref{assu:main} is satisfied. Then for every $\phi \in C_b(\clc_d)$, the collection $\{-\veps^2 \log U^{\veps}[\phi]\}$
	satisfies a large deviation principle on $\RR$ with rate function $I^{\phi}$ and speed $\veps^{-2}$.
\end{theorem}

\section{A Variational Representation}
\label{sec:varrep}
Fix $\phi \in C_b(\clc_d)$.
Recall the functional $U^{\veps}[\phi]$ from \eqref{Uveps}.
From \eqref{Lambdadefn}, note that one can write $U^{\veps}[\phi]$ as
$$U^{\veps}[\phi] = \Lambda^{\veps}_T\left( \exp\{-\veps^{-2} \phi(\cdot)\}, Y^{\veps}\right)$$
whose distribution under $\bar \PP$ is same as the distribution of $\Lambda^{\veps}_T\left( \exp\{-\veps^{-2} \phi(\cdot)\}, \zeta\right)$ under $\PP^{\veps}$. Let 
$$V^{\veps}[\phi] = -\veps^2 \log \Lambda^{\veps}_T\left( \exp\{-\veps^{-2} \phi(\cdot)\}, \zeta\right).$$

Using this equality of laws and the equivalence between Large deviation principles and Laplace principles (see e.g. \cite[Theorems 1.5 and 1.8]{budhiraja2019analysis}), in order to prove Theorem \ref{thm:main} it suffices to show that, $I^{\phi}$ has compact sub-level sets, i.e., 
\begin{equation}
	\label{eq:levset}
	\mbox{ for every } m \in \RR_+,  \{z \in \RR: I^{\phi}(z) \le m\}  \mbox{ is compact,}
\end{equation}
and for every
$G\in C_b(\RR)$
\begin{equation}\label{eq:mainvar}
 \lim_{\veps \to 0}-\veps^2 \log \EE_{\PP^{\veps}}\left[ \exp\left\{ - \veps^{-2} G(V^{\veps}[\phi])\right\}\right]
	= \inf_{z\in \RR}\{ G(z) + I^{\phi}(z)\}.
\end{equation}

The proof of the identity in \eqref{eq:mainvar} will use a variational representation for nonnegative functionals of Brownian motions given by Bou\'{e} and Dupuis\cite{boue1998variational}.
We now use this representation to give a variational formula for the left side of the above equation.
Let $\clf_t$ denote the $\PP^{\veps}$-completion of $\clf^0_t$ and denote by
$\cla^k$ [resp. $\cla^m$] the collection of all $\{\clf_t\}$-progressively measurable $\R^k$ [resp. $\R^m$] valued processes $g$ such that for some $M= M(g)\in (0,\infty)$
$$\int_0^T \|g(s)\|^2 ds \le M \;\; \mbox{a.s.}$$
For $(u,v) \in \cla^k \times \cla^m$, let $\xi^{u}$ be given as the unique solution of the SDE on $(\Om, \clf, \{\clf_t\}, \PP^{\veps})$:
\begin{equation}\label{eq:sigcanonCont}
	\xi^{u}(t) = x_0 + \int_0^t b(\xi^u(s)) ds + \veps\int_0^t \sigma(\xi^u(s)) d \gamma(s) + \int_0^t \sigma(\xi^u(s)) u(s) ds, \; 0 \le t \le T.
\end{equation}
Also define
\begin{equation}
	\zeta^{u,v}(t) = \int_0^t h(\xi^u(s)) ds + \veps \beta(t) + \int_0^t v(s) ds, \; 0 \le t \le T.
\end{equation}
Occasionally, to emphasize the dependence of above processes on $\veps$ we will write $(\xi^u, \zeta^{u,v})$ as $(\xi^{\veps,u}, \zeta^{\veps,u,v})$.

Now let
\begin{equation}
\bar V^{\veps, u,v}[\phi] \doteq -\veps^2 \log \Lambda^{\veps}_T\left( \exp\{-\veps^{-2} \phi(\cdot)\}, \zeta^{\veps,u,v}\right).
\end{equation}
When clear from context we will drop $(u,v, \phi)$ from the notation in $\bar V^{\veps, u,v}[\phi]$ and simply write $\bar V^{\veps}$.
Then it follows from \cite{boue1998variational} (cf. \cite[Theorems 3.17]{budhiraja2019analysis}) that
\begin{equation}\label{eq:repboudup}
	\begin{aligned}
		&
	-\veps^2 \log \EE_{\PP^{\veps}}\left[ \exp\left\{ - \veps^{-2} G(V^{\veps}[\phi])\right\}\right]\\
	&= \inf_{(u,v) \in \cla^k\times \cla^m} \EE_{\PP^{\veps}}\left[G(\bar V^{\veps, u,v}[\phi]) + \frac{1}{2} \int_0^T (\|u(s)\|^2 + \|v(s)\|^2) ds \right].
	\end{aligned}
\end{equation}
\section{A Key Lemma}
\label{sec:keylem}
For $M \in (0, \infty)$, let 
$$S_M \doteq \{(\varphi, \psi) \in \cll^2_k\times \cll^2_m: \int_0^T (\|\varphi(s)\|^2 + \|\psi(s)\|^2) ds \le M\}.$$
We equip, $S_M$ with the weak topology under which $(\varphi_n, \psi_n) \to (\varphi, \psi)$ as $n \to \infty$ if and only if for all $(f,g) \in \cll^2_k\times \cll^2_m$
$$\int_0^T [\langle \varphi_n(s), f(s)\rangle + \langle \psi_n(s), g(s)\rangle] ds \to \int_0^T [\langle \varphi(s), f(s)\rangle + \langle \psi(s), g(s)\rangle] ds$$
as $n\to \infty$. This topology can be metrized so that $S_M$ is a compact metric space.

Recall $\phi \in C_b(\clc_d)$ in the statement of Theorem \ref{thm:main}. For $(\varphi, \psi) \in \cll^2_k\times \cll^2_m$ define
\begin{equation}
	V_0^{\varphi, \psi}[\phi] \doteq \inf_{\eta \in \clc_d} \left[ H(\eta, \xi_0^{\varphi}, \psi) + \phi(\eta) + J(\eta)\right]
	- \inf_{\eta \in \clc_d} \left[ H(\eta, \xi_0^{\varphi}, \psi)  + J(\eta)\right].
\end{equation}
Note that with this notation $\cls(z)$ (introduced below \eqref{eq:clszchA}) is the collection of all $(\varphi, \psi)$ in
$\cll^2_k\times \cll^2_m$ such that
$V_0^{\varphi, \psi}[\phi]=z$.

The following lemma will be the key to the proof of Theorem \ref{thm:main}.
\begin{lemma}\label{lem:keycgce}
	Fix $M \in (0,\infty)$. Let $\{(u_n, v_n)\}$ be a sequence of $S_M$ valued random variables such that $(u_n, v_n) \in \cla^k\times \cla^m$ for every $n$.
	Suppose that $(u_n, v_n)$ converges in distribution to $(u,v)$. Suppose $\veps_n$ is a sequence of positive reals converging to $0$ as $n\to \infty$.
	Then  $\bar V^{\veps_n, u_n,v_n}[\phi] \to V_0^{u, v}[\phi]$, in distribution, as $n\to \infty$.
\end{lemma}
As an immediate corollary of the lemma we have the following.
\begin{corollary}\label{cor:keycgcecor} As $\veps \to 0$, 
	$$-\veps^2 \log U^{\veps}[\phi] \stackrel{ \PP^{\veps}}{\longrightarrow}
	\inf_{\eta \in \clc_d} \left[\phi(\eta)+ \frac{1}{2} \int_0^T \|h(\eta(s)) - h(\xi^*(s))\|^2 ds  + J(\eta)\right].$$
\end{corollary}
\begin{proof}
	The proof follows on observing that, $\bar V^{\veps,0,0}[\phi] = V^{\veps}[\phi]$ which has the same distribution as $-\veps^2 \log U^{\veps}[\phi]$,
	for $\eta \in \clc_d$,
	$$H(\eta, \xi_0^{0}, 0) = \frac{1}{2} \int_0^T \|h(\eta(s)) - h(\xi^*(s))\|^2 ds$$
	and  that
	$$\inf_{\eta \in \clc_d} \left[ H(\eta, \xi_0^{0}, 0)  + J(\eta)\right] =
	\inf_{\eta \in \clc_d} \left[ \frac{1}{2} \int_0^T \|h(\eta(s)) - h(\xi^*(s))\|^2 ds + J(\eta)\right] =0.$$
\end{proof}

\section{Proof of Lemma \ref{lem:keycgce}.}
\label{sec:pflem4.1}
Let $(u,v) \in \cla_k\times \cla_m$. Define canonical coordinate processes on $\Om_x$  as $\hat\xi(\tilde \om) = \tilde \om_1$ and $\hat\gamma(\tilde \om) = \tilde\om_2$, $\tilde \om =(\tilde \om_1, \tilde \om_2) \in \clc_d\times \clc_k$.
Note that
$$\exp \left[ - \veps^{-2}\bar V^{\veps, u,v}[\phi]\right] = \frac{\Gamma^{\veps}_T\left( \exp\{-\veps^{-2} \phi(\cdot)\}, \zeta^{u,v}\right)}{\Gamma^{\veps}_T\left( 1, \zeta^{\veps,u,v}\right)}$$
and for $f \in C_b(\clc_d)$,  $\PP^{\veps}$ a.s.,
\begin{align*}
	\Gamma^{\veps}_T\left( f, \zeta^{\veps,u,v}\right)
	&= \int_{\Omega_x} f(\hat \xi(\tilde \om)) e^{\frac{1}{\veps^2} \int_0^t \langle h(\hat \xi(\tilde \om)(s)), d\zeta^{u,v}(s)\rangle - \frac{1}{2\veps^2} \int_0^t \|h(\hat \xi(\tilde \om(s)))\|^2 ds} \mu^{\veps}(d\tilde \om).
\end{align*}
Suppressing $\tilde \om$ in notation, we have
\begin{align*}
	&\frac{1}{\veps^2} \int_0^t \langle h(\hat \xi(s)), d\zeta^{u,v}(s)\rangle - \frac{1}{2\veps^2} \int_0^t \|h(\hat \xi(s))\|^2 ds\\
	&= \frac{1}{\veps} \int_0^T \langle h(\hat \xi(s)),  d \beta(s)\rangle + \frac{1}{\veps^2} \int_0^T h(\hat \xi(s))\cdot  v(s) ds \\
	&\quad- \frac{1}{2\veps^2} \int_0^T \|h(\hat \xi(s)) - h(\xi^u(s))\|^2 ds+ \frac{1}{2\veps^2} \int_0^T \|h(\xi^u(s))\|^2 ds\\
	&= \frac{1}{\veps} \int_0^T \langle h(\hat \xi(s)), d \beta(s)\rangle - \frac{1}{\veps^2} H(\hat \xi, \xi^u, v)\\
	&\quad +  \frac{1}{2\veps^2} \int_0^T (\|h(\xi^u(s))\|^2 + \|v(s)\|^2) ds
	+ \frac{1}{\veps^2} \int_0^T  h(\xi^u(s) )\cdot v(s) ds.
\end{align*}
Thus, letting
\begin{equation}
	F(\tilde \om, \beta) \doteq \int_0^T \langle h(\hat \xi(\tilde \om)(s)),  d \beta(s)\rangle
\end{equation}
 we can write
\begin{equation}\label{eq:eq606}
e^{ - \veps^{-2} \bar V^{\veps, u,v}[\phi]} = \frac{ \int_{\Omega_x} e^{ \frac{1}{\veps}F(\tilde \om, \beta) -
\frac{1}{\veps^2} ( \phi(\hat \xi(\tilde \om)) + H(\hat \xi(\tilde \om), \xi^u, v))} \mu^{\veps}(d\tilde \om)}
{ \int_{\Omega_x} e^{ \frac{1}{\veps}F(\tilde \om, \beta) -
\frac{1}{\veps^2}  H(\hat \xi(\tilde \om), \xi^u, v)} \mu^{\veps}(d\tilde \om)}.
\end{equation}
Let now $\veps_n, u_n,v_n, u,v$ be as in the statement of Lemma \ref{lem:keycgce}. Using Assumption \ref{assu:main} it is immediate that
\begin{equation}\label{eq:skorrep}
	(u_n, v_n, \xi^{\veps_n,u_n}, \zeta^{\veps_n,u_n,v_n}, \beta) \Rightarrow (u,v, \xi^u_0, \zeta_0^{u,v}, \beta)
\end{equation}
in $S_M\times \clc_d\times \clc_m\times \clc_m$, where
$$\zeta_0^{u,v}(t) = \int_0^t h(\xi_0^u(s)) ds + \int_0^t v(s) ds, \; t \in [0,T].$$
By appealing to Skorohod representation theorem we can obtain,  on some probability space $(\Om^*, \clf^*, \PP^*)$, random variables
$(\tilde u_n, \tilde v_n, \tilde \xi^{n}, \tilde \zeta^{n}, \tilde \beta^n)$ with same law as the random vector on the left side of \eqref{eq:skorrep}
and $(\tilde u,\tilde v, \tilde \xi_0, \tilde \zeta_0, \tilde \beta)$ with same law as the vector on the right side of \eqref{eq:skorrep}, such that 
\begin{equation}\label{eq:ascgce}
	(\tilde u_n, \tilde v_n, \tilde \xi^{n}, \tilde \zeta^{n}, \tilde \beta^n) \to (\tilde u,\tilde v, \tilde \xi_0, \tilde \zeta_0, \tilde \beta),\;  \PP^*-\mbox{ a.s.}
	\end{equation}
Henceforth, to simplify notation we will drop the $\tilde \cdot$ from the notation in the above vectors and denote the corresponding process
$\bar V^{\veps_n, u_n,v_n}[\phi]$ as $\bar V^n[\phi]$. 
Then, from \eqref{eq:eq606}, and the distributional equality noted above, it follows that
	\begin{multline}\label{eq:eq606b}
e^{- \veps_n^{-2} \bar V^{n}[\phi]}  = \frac{ \int_{\Omega_x} e^{\frac{1}{\veps_n}F(\tilde \om, \beta^n) -
\frac{1}{\veps_n^2} ( \phi(\hat \xi(\tilde \om)) + H(\hat \xi(\tilde \om), \xi^n, v^n))} \mu^{\veps_n}(d \tilde \om)}
{ \int_{\Omega_x} e^{ \frac{1}{\veps_n}F(\tilde \om, \beta^n) -
\frac{1}{\veps_n^2}  H(\hat \xi(\tilde \om), \xi^n, v^n)} \mu^{\veps_n}(d \tilde \om)}\\
= \frac{ \int_{\Omega_x} e^{ \frac{1}{\veps_n}F(\tilde \om, \beta^n) -
\frac{1}{\veps_n^2} \left( \phi(\hat \xi(\tilde \om)) + H(\hat \xi(\tilde \om), \xi^n, v)  - \int_0^T h(\hat \xi(\tilde \om)(s)) \cdot (v^n(s)-v(s)) ds\right)} \mu^{\veps_n}(d \tilde \om)}
{ \int_{\Omega_x} e^{\frac{1}{\veps_n}F(\tilde \om, \beta^n) -
\frac{1}{\veps_n^2}  \left(H(\hat \xi(\tilde \om), \xi^n, v)  - \int_0^T h(\hat \xi(\tilde \om)(s)) \cdot (v^n(s)-v(s)) ds\right)} \mu^{\veps_n}(d \tilde \om)}.
\end{multline}
In order to prove Lemma \ref{lem:keycgce} it now suffices to show that, for all $\phi \in C_b(\clc_d)$, as $n\to \infty$,
\begin{multline}\label{eq:vn1phi}
 \bar \Upsilon_1^n[\phi]
\doteq - \veps_n^{-2} \log \left[\int_{\Omega_x}e^{\frac{1}{\veps_n}F(\tilde \om, \beta^n) -
\frac{1}{\veps_n^2} \left( \phi(\hat \xi(\tilde \om)) + H(\hat \xi(\tilde \om), \xi^n, v)  - \int_0^T h(\hat \xi(\tilde \om)(s)) \cdot (v^n(s)-v(s)) ds\right)}\mu^{\veps_n}(d \tilde \om)\right]\\
\to \inf_{\eta \in \clc_d} \left[ H(\eta, \xi_0,  v) + \phi(\eta) + J(\eta)\right] \mbox{ a.s. }  \PP^*.
\end{multline}
Define $\Delta^n_1: \clc_d \times \Om^* \to \RR$ as
\begin{equation}\label{eq:delndef}
	\begin{aligned}
	\Delta^n_1(\eta) &= 
	H(\eta, \xi_0,v)-H(\eta, \xi^n, v) + \int_0^T h(\eta(s)) \cdot (v^n(s)-v(s))ds\\
	&= \frac{1}{2}\int_0^T\left(2\left(h(\eta(s))-v(s)\right)\cdot \left(h(\xi^n(s))-h(\xi_0(s))\right)
	+\|h(\xi_0(s))\|^2-\|h(\xi^n(s))\|^2\right)ds\\
	&\quad +\int_0^Th(\eta(s))\cdot \left(v^n(s)-v(s)\right)ds.
	\end{aligned}
\end{equation}
Then from the continuity of $h$ and the a.s. convergence in \eqref{eq:ascgce}, we see that for every $\eta \in \clc_d$
\begin{equation}\label{eq:delntoz}
	\mbox{ as } n \to \infty,\;\; \Delta_1^n(\eta) \to 0, \mbox{ a.s. }  \PP^*.
\end{equation}
Furthermore, with $\Delta^n(\tilde \om, \om^*) \doteq \Delta_1^n(\hat \xi(\tilde \om), \om^*)$,
\begin{equation}
	\bar \Upsilon_1^n[\phi] = -\veps_n^2 \log \left[\int_{\Omega_x} e^{\frac{1}{\veps_n}F(\tilde \om, \beta^n) -
	\frac{1}{\veps_n^2} \left( \phi(\hat \xi(\tilde \om)) + H(\hat \xi(\tilde \om), \xi_0, v)  - \Delta^n\right)} \mu^{\veps_n}(d \tilde \om)\right].
\end{equation}
In order to prove \eqref{eq:vn1phi} we will show 
	\begin{equation}\label{eq:uppvn1}
	\limsup_{n\to \infty}\bar \Upsilon_1^n[\phi] \le \inf_{\eta \in \clc_d} \left[ H(\eta, \xi_0,  v) + \phi(\eta) + J(\eta)\right] \mbox{ a.s. }  \PP^*
\end{equation}
and
	\begin{equation}\label{eq:lowvn1}
	\liminf_{n\to \infty}\bar \Upsilon_1^n[\phi] \ge \inf_{\eta \in \clc_d} \left[ H(\eta, \xi_0,  v) + \phi(\eta) + J(\eta)\right] \mbox{ a.s. }  \PP^*.
\end{equation}
The fact that $F$ can be neglected in the asymptotic formula follows along the lines of \cite{Heunis1987nonlinear}, however since, unlike \cite{Heunis1987nonlinear}, we do not assume that $h$ is bounded and our functional of interest is somewhat different from the one considered in \cite{Heunis1987nonlinear}, we provide the details.
\subsection{Proof of \eqref{eq:lowvn1}}
\label{sec:pfoflowvn}
We begin with the following lemmas.
\begin{lemma}\label{expmomenttheorem} For any $C\in (0, \infty)$,
	\begin{align*}
		\limsup_{\veps \to 0}\veps^2 \log\int_{\calC_{x}} \exp\left( C\veps^{-2}\|\hat \xi(\tilde \om)\|_*\right)\mu^{\veps}(d\tilde \om)<\infty.
	\end{align*}
\end{lemma}
\begin{proof}
Note that for $t \in [0,T]$
$$\hat \xi(t) = x_0 + \int_0^t b(\hat \xi(s)) ds + \veps \int_0^t \sigma(\hat \xi(s)) d \hat \gamma(s).$$
Let $M(t) \doteq \int_0^t \sigma(\hat \xi(s)) d \hat \gamma(s)$. Then by an application of Gronwall's lemma, it suffices to show that
$$\limsup_{\veps \to 0}\veps^2 \log E_{\mu^{\veps}}\exp\left( C\veps^{-1}\|M\|_*\right) <\infty $$
where $E_{\mu^{\veps}}$ is the expectation under the probability measure $\mu^{\veps}$.
Since $\sigma$ is bounded and under $\mu^{\veps}$, $\hat \gamma$ is a Brownian motion, there is $C_1 \in (0, \infty)$ such that
$E_{\mu^{\veps}}\exp\left( C\veps^{-1}\|M\|_*\right) \le C_1 \exp\left( C_1\veps^{-2}\right)$ for every $\veps>0$.
The result follows.
\end{proof}

\begin{lemma}\label{llimsup}
Let for $\veps>0$, $\bar \clr^{\veps}$ and $\bar A^{\veps}$  be  measurable maps from $\clc_d$ to $\RR$ such that 
\begin{equation}\label{eq:bdaeta}
	\sup_{\veps>0}\bar \clr^{\veps}(\eta) \le c_R (1 + \|\eta\|_*), \; \;
	\sup_{\veps>0}|\bar A^{\veps}(\eta)| \le  c_A(1 + \|\eta\|_{*}) \mbox{ for all }  \eta \in \clc_d. \end{equation}
Then
	\begin{align}\label{eq:eq5125}
	\limsup_{\veps \to 0} \veps^2 \log \int _{\Omega_x}  e^{\veps^{-1} \bar A^{\veps}(\hat\xi(\tilde \om)) + \veps^{-2} \bar\calR^\veps(\hat\xi(\tilde \om))} \mu^\veps (d\tilde \om)
	\leq \limsup_{\veps \to 0} \veps^2 \log \int _{\Omega_x}  e^{\veps^{-2} \bar\calR^\veps(\hat\xi(\tilde \om))} \mu^\veps (d\tilde \om)
	\end{align}
	and for every $c_0 \in (0,\infty)$
	\begin{align}\label{eq:cobd}
	\limsup_{M\to \infty}\limsup_{\veps \to 0} \veps^2 \log \int _{\Omega_x}  e^{\veps^{-1} \bar A^{\veps}(\hat\xi(\tilde \om))+ \veps^{-2} c_0(1+ \|\hat\xi(\tilde \om)\|_*)}
	1_{\{\bar A^{\veps}(\hat\xi(\tilde \om)\ge M\}}
	 \mu^\veps (d\tilde \om) = -\infty .
	\end{align}
\end{lemma}
\begin{proof}
	For $M\in (0,\infty)$, let $A_M^{\veps} \doteq \bar A^{\veps} \wedge M$. Then
\begin{align*}
	\int_{\Omega_x}  e^{\veps^{-1} \bar A^{\veps}(\hat\xi(\tilde \om)) + \veps^{-2} \bar\calR^\veps(\hat\xi(\tilde \om))} \mu^\veps (d\tilde \om)
	&\le \int_{\Omega_x}  e^{\veps^{-1}  A_M^{\veps}(\hat\xi(\tilde \om)) + \veps^{-2} \bar\calR^\veps(\hat\xi(\tilde \om))} \mu^\veps (d\tilde \om)\\
	&\quad + \int_{\Omega_x}  e^{\veps^{-1}  \bar A^{\veps}(\hat\xi(\tilde \om)) + \veps^{-2} \bar\calR^\veps(\hat\xi(\tilde \om))}1_{\{\bar A^{\veps}(\hat\xi(\tilde \om)\ge M\}} \mu^\veps (d\tilde \om).
\end{align*}
Thus
\begin{multline*}
	\limsup_{\veps \to 0} \veps^2 \log \int_{\Omega_x}  e^{\veps^{-1} \bar A^{\veps}(\hat\xi(\tilde \om)) + \veps^{-2} \bar\calR^\veps(\hat\xi(\tilde \om))} \mu^\veps (d\tilde \om)\\
	\le \max\Big\{\limsup_{\veps \to 0} \veps^2 \log \int_{\Omega_x}  e^{\veps^{-1}  A_M^{\veps}(\hat\xi(\tilde \om)) + \veps^{-2} \bar\calR^\veps(\hat\xi(\tilde \om))} \mu^\veps (d\tilde \om),\\
	 \limsup_{\veps \to 0} \veps^2 \log \int_{\Omega_x}  e^{\veps^{-1}  \bar A^{\veps}(\hat\xi(\tilde \om)) + \veps^{-2} \bar\calR^\veps(\hat\xi(\tilde \om))}1_{\{\bar A^{\veps}(\hat\xi(\tilde \om)\ge M\}} \mu^\veps (d\tilde \om)\Big\}.
\end{multline*}
Since
$$\limsup_{\veps \to 0} \veps^2 \log \int_{\Omega_x}  e^{\veps^{-1}  A_M(\hat\xi(\tilde \om)) + \veps^{-2} \bar\calR^\veps(\hat\xi(\tilde \om))} \mu^\veps (d\tilde \om) = \limsup_{\veps \to 0} \veps^2 \log \int_{\Omega_x}  e^{\veps^{-2} \bar\calR^\veps(\hat\xi(\tilde \om))} \mu^\veps (d\tilde \om),$$
in order to prove the lemma it suffices to show \eqref{eq:cobd} for every $c_0 \in (0,\infty)$.
Fix $\veps \in (0,1)$. Using the fact that, on the set $\{\bar A^{\veps}(\hat\xi(\tilde \om))\ge M\}$,
$$ \veps^{-1}  \bar A^{\veps}(\hat\xi(\tilde \om)) \le \veps^{-2}  (\bar A^{\veps}(\hat\xi(\tilde \om))-M) + \veps^{-1} M,$$ 
and the bound in \eqref{eq:bdaeta},
we see that
\begin{align*}
	&\limsup_{\veps \to 0} \veps^2 \log \int_{\Omega_x} e^{\veps^{-1}  \bar A^{\veps}(\hat\xi(\tilde \om))+ \veps^{-2} c_0(1+ \|\hat\xi(\tilde \om)\|_*)} 1_{\{\bar A^{\veps}(\hat\xi(\tilde \om)\ge M\}}  \mu^\veps (d\tilde \om)\\
	&\le -M + \limsup_{\veps \to 0} \veps^2 \log \int_{\Omega_x} e^{\veps^{-2} (c_A+c_0)( 1+ \|\hat\xi(\tilde \om)\|_*)}\mu^\veps (d\tilde \om).
\end{align*}
The inequality in \eqref{eq:cobd} now follows on  applying Lemma \ref{expmomenttheorem}. \hfill
\end{proof}
Note that, by It\^{o}'s formula,
\begin{align*}
F(\tilde \om, \beta^n)&= \int_0^T \langle h(\hat \xi(\tilde \om)(s)),  d \beta^n(s)\rangle\\
&= \langle h(\hat \xi(\tilde \om)(T)), \beta^n(T)\rangle- \sum_{l=1}^m \int_0^T \beta^n_l(s) \langle \nabla h_{l}(\hat \xi(\tilde \om)(s)), d\hat \xi(\tilde \om)(s)\rangle\\
& -\frac{\veps^2}{2}\sum_{i,j=1}^{k}\sum_{l=1}^m\int_0^T \beta^n_l(s)(\sigma\sigma^\dagger)_{ij}(\hat \xi(\tilde \om)(s)) \frac{\partial^2h_l}{\partial x_i \partial x_j}(\hat \xi(\tilde \om)(s))ds\\
&=\langle h(\hat \xi(\tilde \om)(T)), \beta^n(T)\rangle -  \sum_{l=1}^m\int_0^T \beta^n_l(s) \langle\nabla h_l(\hat \xi(\tilde \om)(s)),  b(\hat \xi(\tilde \om)(s))\rangle ds\\
& -\frac{\veps^2}{2}\sum_{i,j=1}^{k} \sum_{l=1}^m \int_0^T \beta^n_l(s)(\sigma\sigma^\dagger)_{ij}(\hat \xi(\tilde \om)(s)) \frac{\partial^2h_l}{\partial x_i \partial x_j}(\hat \xi(\tilde \om)(s))ds\\
&- \sum_{l=1}^m \int_0^T \beta^n_l(s) \left\langle \nabla h_l(\hat \xi(\tilde \om)(s)), \left(d\hat \xi(\tilde \om)(s)- b(\hat \xi(\tilde \om)(s))ds\right)\right\rangle\\
&= A_T(\hat \xi(\tilde \om), \beta^n) + K_T(\hat \xi(\tilde \om), \beta^n),
\end{align*}
where, $\PP^{\veps}$ a.s.,
$$
K_T(\xi, \beta) \doteq 
- \sum_{l=1}^m \int_0^T \beta_l(s) \left\langle \nabla h_l(\xi(s)), \left(d\xi(s)- b(\xi(s))ds\right)\right\rangle$$
and $A_T(\xi, \beta)  = \int_0^T \langle h(\xi(s)),  d \beta(s)\rangle - K_T(\xi, \beta)$.

The following result is taken from Heunis\cite{Heunis1987nonlinear}(cf. page 940 therein).
\begin{proposition}[Heunis\cite{Heunis1987nonlinear}] \label{prop:heunis}
	The maps $K_T$ and  $A_T$ are  measurable and continuous, respectively, from $\clc_d\times \clc_m$ to $\RR$, and
	there are $c_1, c_2 \in (0,\infty)$ such that
for all $x>0$, $n\ge 1$,
\begin{align*}
	\mu^{\veps_n} (\tilde \om: |K_T(\hat\xi(\tilde \om), \beta^n)|>x)\leq 2\exp\left\{- c_1\frac{x^2}{\veps_n^{2} (1 + \|\beta^n\|_*^2)} \right\}, \; \text{ a.s. } \PP^*
	\end{align*}
	and
	\begin{equation}\label{eq:athatbd}
		|A_T(\hat \xi(\tilde \om), \beta^n)| \le c_2 (1 + \|\hat\xi(\tilde \om)\|_* + \|\beta^n\|_*) \mbox{ a.s. } \mu^{\veps_n}\otimes \PP^*.
	\end{equation}
\end{proposition}
Define
\begin{equation}
	G^n(\tilde \om, \om^*) \doteq -\phi(\hat \xi(\tilde \om)) - H(\hat \xi(\tilde \om), \xi_0(\om^*), v(\om^*))  + \Delta^n(\tilde \om, \om^*), \; (\tilde \om, \om^*) \in \Omega_x\times \Om^*.
\end{equation}
\begin{proposition}\label{inftyminus1}
	For any $\delta \in (0, \infty)$, and $\PP^*$ a.e. $\om^*$
	\begin{align*}
	\limsup_{n\to \infty} \veps_n^2 \log \int_{\{|\veps_n K_T(\hat\xi(\tilde \om), \beta^n(\om^*))|>\delta
	\}} e^{\veps_n^{-2}G^n(\tilde \om, \om^*) +  \veps_n^{-1} F(\tilde \om, \beta^n(\om^*))}\mu^{\veps_n} (d\tilde \om)=-\infty,
	\end{align*}
	\begin{align}\label{eq:atbd}
	\limsup_{n\to \infty} \veps_n^2 \log \int_{\{\veps_n K_T(\hat\xi(\tilde \om), \beta^n(\om^*))<-\delta\}} e^{\veps_n^{-2}G^n(\tilde \om, \om^*) +  \veps_n^{-1} A_T(\hat \xi(\tilde \om), \beta^n(\om^*))}\mu^{\veps_n} (d\tilde \om)=-\infty.
	\end{align}
\end{proposition}
\begin{proof}
	Note that on the set $\{\veps_n K_T(\hat\xi(\tilde \om), \beta^n(\om^*))<-\delta\}$
	$$\veps_n^{-2}G^n(\tilde \om, \om^*) +  \veps_n^{-1} F(\tilde \om, \beta^n(\om^*)) \le 
	\veps_n^{-2}(G^n(\tilde \om, \om^*) - \delta) +  \veps_n^{-1} A_T(\hat\xi(\tilde \om), \beta^n(\om^*)).$$
	Also note that, using the linear growth of $h$, one can find a measurable map $\theta: \Om^* \to \RR_+$ such that
	\begin{equation}\label{eq:bdgntil}
		G^n(\tilde \om, \om^*) \le \theta(\om^*) (1+ \|\hat\xi(\tilde \om)\|_*), \mbox{ for all } \tilde \om \in \Omega_x, \; \PP^* \mbox{ a.e. } \om^*.
	\end{equation}
	Using these observations, we have
	\begin{align}
		& \int_{\{\veps_n K_T(\hat\xi(\tilde \om), \beta^n(\om^*))<-\delta\}} e^{\veps_n^{-2}G^n(\tilde \om, \om^*) +  \veps_n^{-1} F(\tilde \om, \beta^n(\om^*))}\mu^{\veps_n} (d\tilde \om)\nonumber\\
		&\le e^{\veps_n^{-2}(\theta(\om^*) - \delta)}\int_{\{\veps_n K_T(\hat\xi(\tilde \om), \beta^n(\om^*))<-\delta\}} e^{\veps_n^{-2}\theta(\om^*)\|\hat\xi(\tilde \om)\|_*+  \veps_n^{-1} A_T(\hat\xi(\tilde \om), \beta^n(\om^*))}\mu^{\veps_n} (d\tilde \om). \label{eq:eq109}
	\end{align}
Next, for every $M \in (0,\infty)$
\begin{align}
	&\int_{\{\veps_n K_T(\hat\xi(\tilde \om), \beta^n(\om^*))<-\delta\}} e^{\veps_n^{-2}\theta(\om^*)\|\hat\xi(\tilde \om)\|_*+  \veps_n^{-1} A_T(\hat\xi(\tilde \om), \beta^n(\om^*))}\mu^{\veps_n} (d\tilde \om)\nonumber\\
	&\le \int_{\{\veps_n K_T(\hat\xi(\tilde \om), \beta^n(\om^*))<-\delta\}} e^{\veps_n^{-2}\theta(\om^*)\|\hat\xi(\tilde \om)\|_*+  \veps_n^{-1} M}\mu^{\veps_n} (d\tilde \om)\nonumber\\
	&\quad + \int_{\{\veps_n K_T(\hat\xi(\tilde \om), \beta^n(\om^*))<-\delta\}} e^{\veps_n^{-2}\theta(\om^*)\|\hat\xi(\tilde \om)\|_*+  \veps_n^{-1} A_T(\hat\xi(\tilde \om), \beta^n(\om^*))} 1_{ \{A_T(\hat\xi(\tilde \om), \beta^n(\om^*))\ge M\}}\mu^{\veps_n} (d\tilde \om).\label{eq:eq1247}
\end{align}
We now consider the two terms in the above display separately.
For the first term, from Cauchy-Schwarz inequality, 
\begin{align*}
	&\int_{\{\veps_n K_T(\hat\xi(\tilde \om), \beta^n(\om^*))<-\delta\}} e^{\veps_n^{-2}\theta(\om^*)\|\hat\xi(\tilde \om)\|_*}\mu^{\veps_n} (d\tilde \om) \\
	&\le
	\left[\int_{\Omega_x} e^{2\veps_n^{-2}\theta(\om^*)\|\hat\xi(\tilde \om)\|_*}\mu^{\veps_n} (d\tilde \om)\right]^{1/2}
	\left[\mu^{\veps_n}\{\veps_n K_T(\hat\xi(\tilde \om), \beta^n(\om^*))<-\delta\}\right]^{1/2}
\end{align*}
and therefore
\begin{align*}
&\limsup_{n\to \infty}\veps_n^2 \log	\int_{\{\veps_n K_T(\hat\xi(\tilde \om), \beta^n(\om^*))<-\delta\}} e^{\veps_n^{-2}\theta(\om^*)\|\hat\xi(\tilde \om)\|_*}\mu^{\veps_n} (d\tilde \om) \\
&\le \limsup_{n\to \infty}\frac{\veps_n^2}{2} \log\int_{\Omega_x} e^{2\veps_n^{-2}\theta(\om^*)\|\hat\xi(\tilde \om)\|_*}\mu^{\veps_n} (d\tilde \om) +
\limsup_{n\to \infty}\frac{\veps_n^2}{2}  \log \mu^{\veps_n}\{ K_T(\hat\xi(\tilde \om), \beta^n(\om^*))<\frac{-\delta}{\veps_n}\}\\
&\le \limsup_{n\to \infty}\frac{\veps_n^2}{2}  \log\int_{\Omega_x} e^{2\veps_n^{-2}\theta(\om^*)\|\hat\xi(\tilde \om)\|_*}\mu^{\veps_n} (d\tilde \om) 
 - c_1\frac{\delta^2}{2\veps_n^{2} (1 + \|\beta^n\|_*^2)} \\
 &= - \infty
\end{align*}
where in the next to last line we have used Proposition \ref{prop:heunis} and in the last line we have appealed to Lemma \ref{expmomenttheorem} and the fact that $\sup_n\|\beta^n\|_*<\infty$ $\PP^*$ a.s.

For the second term on the right side in \eqref{eq:eq1247}, we have from Lemma \ref{llimsup} (see \eqref{eq:cobd})
and \eqref{eq:athatbd}
that
	\begin{align*}
	\limsup_{M\to \infty}\limsup_{n\to \infty} \veps_n^2 \log \int _{\Omega_x}  e^{ \veps_n^{-2} \theta(\om^*)\|\hat\xi(\tilde \om)\|_*+ \veps_n^{-1}  A_T(\hat\xi(\tilde \om), \beta^n(\om^*))}
	1_{\{A_T(\hat\xi(\tilde \om), \beta^n(\om^*)\ge M\}}
	 \mu^\veps (d\tilde \om) = -\infty.
	\end{align*}
Using the last two displays in \eqref{eq:eq1247} and combining with \eqref{eq:eq109} we have \eqref{eq:atbd} and 
	\begin{align*}
	\limsup_{n\to \infty} \veps_n^2 \log \int_{\{\veps_n K_T(\hat\xi(\tilde \om), \beta^n(\om^*))<-\delta\}} e^{\veps_n^{-2}G^n(\tilde \om, \om^*) +  \veps_n^{-1} F(\tilde \om, \beta^n(\om^*))}\mu^{\veps_n} (d\tilde \om)=-\infty.
	\end{align*}
%
%
%
	Next, from \cite[Proposition $4.7$]{Heunis1987nonlinear}, it follows that
	\begin{align*}
	\limsup_{n\to \infty} \veps_n^2 \log \int_{\{\veps_n K_T(\hat\xi(\tilde \om), \beta^n(\om^*))>\delta\}} e^{  2\veps_n^{-1} K_T(\hat\xi(\tilde \om), \beta^n(\om^*))}\mu^{\veps_n} (d\tilde \om)=-\infty.
	\end{align*}
	Now using Cauchy-Schwarz inequality and arguing as before, we see that
		\begin{align*}
		\limsup_{n\to \infty} \veps_n^2 \log \int_{\{\veps_n K_T(\hat\xi(\tilde \om), \beta^n(\om^*))>\delta\}} e^{\veps_n^{-2}G^n(\tilde \om, \om^*) +  \veps_n^{-1} F(\tilde \om, \beta^n(\om^*))}\mu^{\veps_n} (d\tilde \om)=-\infty.
		\end{align*}
	 We omit the details.

\end{proof}

The following proposition shows that the term involving $F$ in the definition of $\bar \Upsilon_1^n[\phi] $ can be ignored in proving the bound in \eqref{eq:lowvn1}.

\begin{proposition}\label{limsup} For $\PP^*$ a.e. $\om^*$,
	\begin{align*}
	\limsup_{n\to \infty} \veps_n^2 \log\int_{\Omega_x} e^{\veps_n^{-2}G^n(\tilde \om, \om^*) +  \veps_n^{-1} F(\tilde \om, \beta^n(\om^*))}\mu^{\veps_n} (d\tilde \om) 
	\leq \limsup_{n\to \infty} \veps_n^2 \log\int_{\Omega_x} e^{\veps_n^{-2}G^n(\tilde \om, \om^*) }\mu^{\veps_n} (d\tilde \om).
	\end{align*}
\end{proposition}
\begin{proof} 
	Fix $\delta \in (0,\infty)$ and write
	\begin{align*}
	&\int_{\Omega_x} e^{\veps_n^{-2}G^n(\tilde \om, \om^*) +  \veps_n^{-1} F(\tilde \om, \beta^n(\om^*))}\mu^{\veps_n} (d\tilde \om) 	\\
	&= \int_{\{\veps_n K_T(\hat\xi(\tilde \om), \beta^n(\om^*))>\delta\}} e^{\veps_n^{-2}G^n(\tilde \om, \om^*) +  \veps_n^{-1} F(\tilde \om, \beta^n(\om^*))}\mu^{\veps_n} (d\tilde \om) \\
	&\quad + \int_{\{\veps_n K_T(\hat\xi(\tilde \om), \beta^n(\om^*))\le\delta\}} e^{\veps_n^{-2}G^n(\tilde \om, \om^*) +  \veps_n^{-1} F(\tilde \om, \beta^n(\om^*))}\mu^{\veps_n} (d\tilde \om).
	\end{align*}
	From Proposition \ref{inftyminus1}, 
	\begin{equation}\label{eq:neg840}
		\limsup_{n\to \infty} \veps_n^2 \log\int_{\{\veps_n K_T(\hat\xi(\tilde \om), \beta^n(\om^*))>\delta\}} e^{\veps_n^{-2}G^n(\tilde \om, \om^*) +  \veps_n^{-1} F(\tilde \om, \beta^n(\om^*))}\mu^{\veps_n} (d\tilde \om) = -\infty.
	\end{equation}
Next note that	
\begin{align*}
	&\int_{\{\veps_n K_T(\hat\xi(\tilde \om), \beta^n(\om^*))\le\delta\}} e^{\veps_n^{-2}G^n(\tilde \om, \om^*) +  \veps_n^{-1} F(\tilde \om, \beta^n(\om^*))}\mu^{\veps_n} (d\tilde \om)\\
	&\le \int_{\{\veps_n K_T(\hat\xi(\tilde \om), \beta^n(\om^*))\le\delta\}} e^{\veps_n^{-2}G^n(\tilde \om, \om^*) +  \delta \veps_n^{-2} +  \veps_n^{-1}A_T(\hat\xi(\tilde \om), \beta^n(\om^*))}\mu^{\veps_n} (d\tilde \om).
\end{align*}
Now recalling \eqref{eq:athatbd} and \eqref{eq:bdgntil} and applying the first inequality in 
Lemma \ref{llimsup} (i.e. \eqref{eq:eq5125}),  we get
\begin{align*}
	&\limsup_{n\to \infty} \veps_n^2 \log\int_{\{\veps_n K_T(\hat\xi(\tilde \om), \beta^n(\om^*))\le\delta\}} e^{\veps_n^{-2}G^n(\tilde \om, \om^*) +  \veps_n^{-1} F(\tilde \om, \beta^n(\om^*))}\mu^{\veps_n} (d\tilde \om)\\
	&\le \delta + \limsup_{n\to \infty} \veps_n^2 \log\int_{\Omega_x} e^{\veps_n^{-2}G^n(\tilde \om, \om^*)  }\mu^{\veps_n} (d\tilde \om).
\end{align*}
Since $\delta>0$ is arbitrary, the result follows on combining the above with \eqref{eq:neg840}.
\end{proof}
The proof of the following lemma follows along the lines of Varadhan's lemma (cf. \cite[Theorem 2.6]{stroock2012introduction}, \cite[Theorem 1.18 ]{budhiraja2019analysis}). We provide details for reader's convenience.
\begin{lemma}\label{unifconvlim} Let $\{Z^\veps\}_{\veps>0}$ be a collection of random variables with values in a Polish space $(\calX, d(\cdot, \cdot))$ that satisfies a LDP with  rate function $J$ and speed $\veps^{-2}$. Let $\phi:\calX\rightarrow \mathbb{R}$ be a  continuous function bounded from above, namely $\sup_{x\in \calX} \phi(x) <\infty$, and let $\{\psi^\veps\}_{\veps>0}$ be a collection of real measurable maps  on $\calX$ such that $\sup_{\veps>0} \sup_{x \in \calX} |\psi^\veps(x)| < \infty$.  Further suppose that 
	\begin{center}
		\begin{align*}
		\text{ for every $\delta>0$ and $x\in \calX$},&\text{ there exist $\veps_0(x)$, $\delta_1(x)\in (0,\infty)$ such that } |\psi^\veps(y)|<\delta,\;\\
		& \mbox{ for all }  d(x,y)<\delta_1(x) \text{ and  all } 0< \veps<\veps_0(x).
		\end{align*}
	\end{center}
	Then  
	\begin{align*}  
	\lim_{\veps \to 0}\veps^2\log\mathbb{E}[\exp\left(\veps^{-2}\left\{\phi(Z^\veps)+\psi^\veps(Z^\veps)\right\}\right)]=\sup_{x\in \calX}\left[\phi(x)-J(x)\right].
	\end{align*}
\end{lemma}
\begin{proof} 
	Define $R\doteq \sup_{x\in \calX}(\phi(x)+\sup_{\veps>0}|\psi^\veps(x)|)$, 
	$S\doteq \sup_{x\in \calX}(\phi(x)-J(x))$ and $K\doteq \{x\in \calX: J(x)\leq |S|+R\}$. 
	Since $J$ is a rate function, $K$ is a compact subset of $\calX$.
	
	Fix $\delta \in (0,\infty)$. From the hypothesis of the lemma, for each $x\in \calX$, there exist $\delta_1(x), \veps_0(x) \in (0, \infty)$ such that $ |\psi^\veps(y)|<\delta$ for every $y\in B(x,\delta_1(x))$ and $\veps \in (0, \veps_0(x))$,
	where $B(z,\gamma) \doteq \{x\in \calX: d(x,z)<\gamma\}$ is an open ball of radius $\gamma$ in $\calX$.
	Also, from the continuity of $\phi$, for every $x\in \calX$ there exists $\delta_2(x)\in (0,\infty)$ such that
	\begin{align*}
	 |\phi(x)-\phi(y)|<\delta,\; \forall y\in B(x,\delta_2(x)). 
	\end{align*}
	Next, from the lower semi-continuity of $J$, for every $x\in\calX$, there exists $\delta_3(x) \in (0,\infty)$ such that 
	\begin{align*}
	J(x)\leq \inf_{y\in \overline{B(x,\delta_3(x))}}J(y)+\delta.
	\end{align*}
	Let $\bar \delta(x) \doteq \min\{\delta_1(x), \delta_2(x), \delta_3(x)\}$.
	Now define an open cover $\cup_{x\in K}U(x)$ of $K$ using the following open sets: 
	\begin{align*}
	U(x)\doteq B(x,\bar\delta(x)),\;  x\in K.
	\end{align*}
	Note that for any, $x\in K$, $y\in U(x)$ and $\veps<\veps_0(x)$, we have
	\begin{align}\label{eq:threefacts}
	|\psi^\veps(y)|<\delta ,\text{  } |\phi(x)-\phi(y)|<\delta \text{ and } J(x)\leq \inf_{z\in \overline{U(x)}}J(z)+\delta.
	\end{align}
	Since $K$ is compact, there exists $N \in \NN$ and $\{x_i\}_{i=1}^N \subset K$ such that $\{U_i\doteq U(x_i)\}_{i=1}^N$cover $K$. For $i=1, \ldots , N$, we can find $0< \veps(x_i) \le \veps_0(x_i)$  such that with 
	$\bar\veps_0\doteq \min_{i=1,\ldots,N}\veps(x_i)$, for every $\veps<\bar \veps_0$,
	\begin{align}\label{eq:uif}
	\mathbb{P}[Z^\veps\in \overline{U_i}]\leq \exp\left[\veps^{-2}(-b_i+\delta)\right],\;\;\;
	\mathbb{P}[Z^\veps\in F]\leq \exp\left[\veps^{-2}(-\inf_{x\in F}J(x) +\delta)\right]
	\end{align}
	where, $F\doteq \left(\cup_{i=1}^N{U_i}\right)^c$ and $b_i\doteq \inf_{x\in \overline{U_i}}J(x)$.

	Next note that
	\begin{align}
	\mathbb{E}[\exp\left(\veps^{-2}\left\{\phi(Z^\veps)+\psi^\veps(Z^\veps)\right\}\right)]&= \mathbb{E}[\exp\left(\veps^{-2}\left\{\phi(Z^\veps)+\psi^\veps(Z^\veps)\right\}\right)1_{\cup_{i=1}^N{U_i}}(Z^\veps)]\nonumber\\
	&\quad +\mathbb{E}[\exp\left(\veps^{-2}\left\{\phi(Z^\veps)+\psi^\veps(Z^\veps)\right\}\right)1_{F}(Z^\veps)]\nonumber\\
	&\leq \sum_{i=1}^N \mathbb{E}[\exp\left(\veps^{-2}\left\{\phi(Z^\veps)+\psi^\veps(Z^\veps)\right\}\right)1_{{U}_i}(Z^\veps)]\nonumber\\ &\quad+\mathbb{E}[\exp\left(\veps^{-2}\left\{\phi(Z^\veps)+\psi^\veps(Z^\veps)\right\}\right)1_{F}(Z^\veps)].\label{eq:eq155}
\end{align}
	Defining $a_i\doteq \inf_{x\in \overline{U_i}}\phi(x)$, we have 
	$|a_i-\phi(x)|<2\delta$, for $x\in {U_i}$. 
Thus, using \eqref{eq:threefacts}  and \eqref{eq:uif}
\begin{align}
&\limsup_{\veps \to 0}\veps^2\log\mathbb{E}[\exp\left(\veps^{-2}\left\{\phi(Z^\veps)+\psi^\veps(Z^\veps)\right\}\right)
1_{{U}_i}(Z^\veps)]\nonumber\\
&\le (a_i-b_i+4\delta) \le \phi(x_i) - J(x_i) +5 \delta \le \sup_{x\in \calX}\left[\phi(x)-J(x)\right] + 5\delta. \label{eq:uibd}
\end{align}
Also
\begin{align}
&\limsup_{\veps \to 0}\veps^2\log\mathbb{E}[\exp\left(\veps^{-2}\left\{\phi(Z^\veps)+\psi^\veps(Z^\veps)\right\}\right)1_{F}] \le R - \inf_{x\in F}J(x) +\delta\nonumber\\
&\quad\le -|S|+\delta \le \sup_{x\in \calX}\left[\phi(x)-J(x)\right] +\delta,\label{eq:fbd}
\end{align}
where the second inequality is a consequence of the observation that $F \subset K^c$.
Since $\delta>0$ is arbitrary, using \eqref{eq:uibd} and \eqref{eq:fbd} in \eqref{eq:eq155} we now see that
\begin{equation}\label{eq:eq206}
\limsup_{\veps \to 0}\veps^2\log	\mathbb{E}[\exp\left(\veps^{-2}\left\{\phi(Z^\veps)+\psi^\veps(Z^\veps)\right\}\right)]
\le \sup_{x\in \calX}\left[\phi(x)-J(x)\right].
\end{equation}
For the lower bound, choose $x_0$ such that $\phi(x_0)-J(x_0)\geq S-\delta$. 
Let $\delta(x_0),  \veps(x_0)\in (0,\infty)$ be such that for all 
 $x\in U\doteq B(x_0,\delta(x_0))$,  $|\phi(x)-\phi(x_0)|<\delta$ and $|\psi^\veps(x)|<\delta$, for $\veps<\veps(x_0)$. Then
	\begin{align*}
	&\liminf_{\veps \to 0}\veps^2\log\mathbb{E}\left[\exp\left(\veps^{-2}\left\{\phi(Z^\veps)+\psi^\veps(Z^\veps)\right\}\right)\right]\\
	&\geq \liminf_{\veps \to 0}\veps^2\log\mathbb{E}\left[\exp\left(\veps^{-2}\left\{\phi(Z^\veps)+\psi^\veps(Z^\veps)\right\}\right)1_{U}(Z^\veps)\right]\\
	&\geq \phi(x_0)-2\delta + \liminf_{\veps \to 0}\veps^2\log \mathbb{P}\left[Z^\veps \in U\right]\\
	&\geq \phi(x_0)-2\delta-\inf_{x\in U}J(x)
	\geq \phi(x_0)-2\delta-J(x_0)
	 \geq \sup_{x\in \calX}\left[\phi(x)-J(x)\right]-3\delta.
	\end{align*} 
	Sending $\delta \to 0$ we have the lower bound and combining it with the upper bound in \eqref{eq:eq206}, we have the result.
\end{proof}

Recall the definition of $\Delta^n_1$ from \eqref{eq:delndef}. 
The following lemma will allow us to apply Lemma \ref{unifconvlim}.
\begin{lemma}\label{uniformconv} For $\PP^*$ a.e. $\om^*$ and every $\delta \in (0,\infty)$ and $\eta \in \clc_d$ there exist $n_0 \in \NN$ and $\delta_1 \in (0,\infty)$ such that
	$$|\Delta^n_1(\tilde \eta)|<\delta \mbox{ whenever } \tilde \eta \in \clc_d, \; \|\eta - \tilde \eta\|_*\le \delta_1 \mbox{ and } n \ge n_0.$$
\end{lemma}
\begin{proof}
	Consider $\om^*$ in the set of full measure on which the convergence in \eqref{eq:ascgce} (and thus in \eqref{eq:delntoz}) holds.
	From \eqref{eq:delntoz}, for any fixed  $\delta \in (0,\infty)$ and $\eta \in \clc_d$, we can find $n_0 \in \NN$ such that for all $n \ge n_0$
	\begin{equation}
		|\Delta^n_1(\eta, \om^*)| \le \frac{\delta}{2}.
	\end{equation}
	Also, from continuity of $h$, we can find a  $\delta_1 \in (0,\infty)$ such that for all  $\tilde \eta \in \clc_d$ with $\|\eta - \tilde \eta\|_* \le \delta_1$
	\begin{equation*}\sup_{n\in \NN} \int_0^T \|h(\eta(s)) - h(\tilde \eta(s))\| (\|h(\xi^n(s)) \| + \|h(\xi_0(s))\|) ds \le \frac{\delta}{4}\end{equation*}
		and
		\begin{equation*}\sup_{n\in \NN} \int_0^T \|h(\eta(s)) - h(\tilde \eta(s))\| (\|v^n(s)\| + \|v(s)\|) ds \le \frac{\delta}{4}.\end{equation*}
			Thus for all $n \ge n_0$ and $\tilde \eta\in \clc_d$ with $\|\eta - \tilde \eta\|_* \le \delta_1$
\begin{align*}
	|\Delta^n_1(\tilde \eta)| &\le |\Delta^n_1(\tilde \eta) -\Delta^n_1(\eta)| + |\Delta^n_1(\eta)|\\
	& \le \int_0^T \|h(\eta(s)) - h(\tilde \eta(s))\| (\|h(\xi^n(s)) \| + \|h(\xi_0(s))\|) ds\\
	&\quad +  \int_0^T \|h(\eta(s)) - h(\tilde \eta(s))\| (\|v^n(s)\| + \|v(s)\|) ds + \frac{\delta}{2}
	\le \delta.
\end{align*}

\end{proof}

We now complete the proof of \eqref{eq:lowvn1}.\\

\noindent {\bf Completing the proof of \eqref{eq:lowvn1}.}
Note that, from Proposition \ref{limsup}, $\PP^*$ a.s.,
\begin{align}
	\limsup_{n\to \infty}-\bar \Upsilon_1^n[\phi]
	&= \limsup_{n\to \infty} \veps_n^2 \log\int_{\Omega_x} e^{\veps_n^{-2}G^n(\tilde \om, \om^*) +  \veps_n^{-1} F(\tilde \om, \beta^n(\om^*))}\mu^{\veps_n} (d\tilde \om) \nonumber\\
	&\leq \limsup_{n\to \infty} \veps_n^2 \log\int_{\Omega_x} e^{\veps_n^{-2}G^n(\tilde \om, \om^*) }\mu^{\veps_n} (d\tilde \om).\label{eq:eq505}
\end{align}

For $Q\in (0,\infty)$, let $\Delta^{n,Q} \doteq (\Delta^n \wedge Q)\vee (-Q)$. Then
\begin{align}
	\int_{\Omega_x} e^{\veps_n^{-2}G^n(\tilde \om, \om^*)} \mu^{\veps_n} (d\tilde \om) &\le 
	\int_{\Omega_x} e^{\veps_n^{-2}G^n(\tilde \om, \om^*) }1_{\{|\Delta^n|\ge Q\}} \mu^{\veps_n} (d\tilde \om)\nonumber\\
	 &\quad + \int_{\Omega_x} e^{\veps_n^{-2}\left(-\phi(\hat \xi(\tilde \om)) - H(\hat \xi(\tilde \om), \xi_0(\om^*), v(\om^*))  + \Delta^{n,Q}(\om, \om^*)\right)} \mu^{\veps_n} (d\tilde \om).\label{eq:eq506}
\end{align}
Note that $\phi$ is a continuous and bounded map on $\clc_d$, $\eta \mapsto H(\eta, \xi_0, v)$ is a 
continuous, nonnegative map on $\clc_d$ and $\eta \mapsto \Delta^n_1(\eta, \om^*) \wedge Q \vee (-Q)$ 
is a   map uniformly bounded in $n$ which satisfies the properties in Lemma \ref{uniformconv}. Thus applying Lemma \ref{unifconvlim} and the large deviations result from \eqref{eq:ldpsign}, we have
\begin{equation}\label{eq:eq507}
	\begin{aligned}
	&\limsup_{n\to \infty} \veps_n^2 \log\int_{\Omega_x} e^{\veps_n^{-2}\left(-\phi(\hat \xi(\tilde \om)) - H(\hat \xi(\tilde \om), \xi_0(\om^*), v(\om^*))  + \Delta^{n,Q}(\om, \om^*)\right)} \mu^{\veps_n} (d\tilde \om)\\
	& \le -\inf_{\eta \in \clc_d} \left[ H(\eta, \xi_0,  v) + \phi(\eta) + J(\eta)\right].
	\end{aligned}
\end{equation}
Next, using the linear growth property of $h$
$$\sup_n|\Delta^n_1(\eta)| \le c_{\Delta}(\om^*) (1 + \|\eta\|_*), \; \PP^* \mbox{ a.s. } $$
for some measurable map $c_{\Delta}: \Om^* \to \RR_+$. Thus, using the boundedness of $\phi$ and the nonnegativity of $H$,  we have
\begin{align*}
	&\limsup_{Q\to \infty}\limsup_{n\to \infty} \veps_n^2 \log\int_{\Omega_x} e^{\veps_n^{-2}G^n(\tilde \om, \om^*) }1_{\{|\Delta^n|\ge Q\}} \mu^{\veps_n} (d\tilde \om)\\
	&\le \limsup_{Q\to \infty}\limsup_{n\to \infty} \veps_n^2 \log\int_{\Omega_x} e^{\veps_n^{-2} (c_{\Delta} +\|\phi\|_{\infty})(1 + \|\hat \xi(\tilde \om)\|_*) }1_{\{c_{\Delta}(1 + \|\hat \xi(\tilde \om)\|_*)\ge Q\}} \mu^{\veps_n} (d\tilde \om) = - \infty
\end{align*}
where the last equality follows from Lemma \ref{llimsup} (see \eqref{eq:cobd}). Using the last bound together with \eqref{eq:eq507} in \eqref{eq:eq506} and \eqref{eq:eq505} we now have the inequality in \eqref{eq:lowvn1}. \hfill \qed

\subsection{Proof of \eqref{eq:uppvn1}}
\label{sec:pfofuppvn}
Recall the convergence from \eqref{eq:ascgce}.
We begin with the following lemma.
\begin{lemma}\label{lliminf} For $\PP^*$ a.e. $\om^*$
	\begin{align*}
	\liminf_{n\to \infty} \veps_n^2 \log \int _{\Omega_x} e^{\veps_n^{-2}G^n(\tilde \om, \om^*) +  \veps_n^{-1} A_T(\hat \xi(\tilde \om), \beta^n(\om^*))} \mu^\veps (d\tilde \om)&\geq -\inf_{\eta\in \calC_d}\left[ H(\eta, \xi_0,  v) + \phi(\eta)+J(\eta)\right].\\
	\end{align*}
\end{lemma}
\begin{proof} 
	Fix $\eta_0 \in \clc_d$ and $\delta \in (0,\infty)$. From continuity of $\phi$ on $\clc_d$, of  $A_T$ on $\clc_d\times \clc_m$, and of $\eta \mapsto H(\eta, \xi_0(\om^*), v(\om^*))$  (for $\PP^*$ a.e. $\om^*$)  on $\clc_d$, a.s. convergence of $\beta^n$ to $\beta$,
	and Lemma \ref{uniformconv},
	we can find, for  $\PP^*$ a.e. $\om^*$, a neighbourhood $G$ of $\eta_0$ and $n_1 \in \NN$ such that
	\begin{align*}
	&\inf_{\tilde \eta\in G} A_T(\tilde \eta, \beta^n(\om^*))\geq A_T(\eta_0, \beta^n(\om^*))-\delta,  \text{  for all  }  n \ge n_1,\\
	&\quad \inf_{\tilde \eta\in G}  [-\phi(\tilde \eta) - H(\tilde \eta, \xi_0(\om^*), v(\om^*)]\geq [-\phi(\eta_0) - H(\eta_0, \xi_0(\om^*), v(\om^*)]-\delta,\\
	&\quad \sup_{\tilde \eta\in G} |\Delta^n_1(\tilde \eta)|<\delta \mbox{ for all  }  n \ge n_1.
	\end{align*}
	Observe that
	\begin{align*}
		&\int _{\Omega_x} e^{\veps_n^{-2}G^n(\tilde \om, \om^*) +  \veps_n^{-1} A_T(\hat \xi(\tilde \om), \beta^n(\om^*))} \mu^\veps (d\tilde \om)\\
		&\ge \int _{\Omega_x} e^{\veps_n^{-2}G^n(\tilde \om, \om^*) +  \veps_n^{-1} A_T(\hat \xi(\tilde \om), \beta^n(\om^*))} 1_{\{\hat \xi(\tilde \om) \in G\}} \mu^\veps (d\tilde \om)\\
		&\ge   e^{\veps_n^{-2} [-\phi(\eta_0) - H(\eta_0, \xi_0(\om^*), v(\om^*) -2\delta] +  \veps_n^{-1} (A_T(\eta_0, \beta^n(\om^*))-\delta)}  \mu^\veps (G),
	\end{align*}
	Noting that $\sup_n|A_T(\eta_0, \beta^n(\om^*))|<\infty$ $\PP^*$ a.s. and applying the large deviation result from \eqref{eq:ldpsign}, we now have
\begin{align*}
	&\liminf_{n\to \infty} \veps_n^2 \log\int _{\Omega_x} e^{\veps_n^{-2}G^n(\tilde \om, \om^*) +  \veps_n^{-1} A_T(\hat \xi(\tilde \om), \beta^n(\om^*))} \mu^\veps (d\tilde \om)\\
	&\ge  [-\phi(\eta_0) - H(\eta_0, \xi_0(\om^*), v(\om^*) -2\delta] - \inf_{\tilde \eta \in G} J(\tilde \eta)\\
	&\ge  -\phi(\eta_0) - H(\eta_0, \xi_0(\om^*), v(\om^*) -  J(\eta_0) - 2\delta.
\end{align*}
	Since $\delta \in (0, \infty)$ and $\eta_0 \in \clc_d$ are arbitrary, the result follows.
\end{proof}

We now complete the proof of \eqref{eq:uppvn1}.\\

\noindent {\bf Completing the proof of \eqref{eq:uppvn1}.}
Fix $\delta \in (0,\infty)$. Then 
\begin{align*}
	&\int_{\Omega_x} e^{\veps_n^{-2}G^n(\tilde \om, \om^*) +  \veps_n^{-1} F(\tilde \om, \beta^n(\om^*))}\mu^{\veps_n} (d\tilde \om) \\
	 &\ge \int_{\{\veps_n K_T(\hat\xi(\tilde \om), \beta^n(\om^*))\ge-\delta\}} e^{\veps_n^{-2}G^n(\tilde \om, \om^*) +  \veps_n^{-1} F(\tilde \om, \beta^n(\om^*))}\mu^{\veps_n} (d\tilde \om) \\
	 &\ge \int_{\{\veps_n K_T(\hat\xi(\tilde \om), \beta^n(\om^*))\ge-\delta\}} e^{\veps_n^{-2}G^n(\tilde \om, \om^*) +  \veps_n^{-1} (- \delta \veps_n^{-1} +  A_T(\hat \xi(\tilde \om), \beta^n(\om^*))}\mu^{\veps_n} (d\tilde \om) \\
	 &= \int_{\Omega_x} e^{\veps_n^{-2}(G^n(\tilde \om, \om^*)-\delta) +  \veps_n^{-1}   A_T(\hat \xi(\tilde \om), \beta^n(\om^*)}\mu^{\veps_n} (d\tilde \om)\\
	 &\quad - \int_{\{\veps_n K_T(\hat\xi(\tilde \om), \beta^n(\om^*))<-\delta\}} e^{\veps_n^{-2}(G^n(\tilde \om, \om^*)-\delta) +  \veps_n^{-1}  A_T(\hat \xi(\tilde \om), \beta^n(\om^*)}\mu^{\veps_n} (d\tilde \om).
\end{align*}
From Proposition \ref{inftyminus1} (see \eqref{eq:atbd})
	\begin{align*}
	\limsup_{n\to \infty} \veps_n^2 \log \int_{\{\veps_n K_T(\hat\xi(\tilde \om), \beta^n(\om^*))<-\delta\}} e^{\veps_n^{-2}G^n(\tilde \om, \om^*) +  \veps_n^{-1} A_T(\hat \xi(\tilde \om), \beta^n(\om^*))}\mu^{\veps_n} (d\tilde \om)=-\infty.
	\end{align*}
	Thus to prove \eqref{eq:uppvn1} it suffice to show that, $ \PP^*$ a.s., 
	\begin{equation}
		\liminf_{n\to \infty}  \veps_n^2 \log \int_{\Omega_x} e^{\veps_n^{-2} G^n(\tilde \om, \om^*) +  \veps_n^{-1}   A_T(\hat \xi(\tilde \om), \beta^n(\om^*)}\mu^{\veps_n} (d\tilde \om) \ge -\inf_{\eta \in \clc_d} \left[ H(\eta, \xi_0,  v) +\phi(\eta) + J(\eta)\right] .
	\end{equation}
	However the above is an immediate consequence of Lemma \ref{lliminf}. This completes the proof of \eqref{eq:uppvn1}. \hfill \qed\\

Finally we complete the proof of Lemma \ref{lem:keycgce}.\\

{\bf Completing the proof of Lemma \ref{lem:keycgce}.}

As noted above \eqref{eq:vn1phi}, in order to prove Lemma \ref{lem:keycgce} it suffices to show \eqref{eq:vn1phi} for every $\phi \in C_b(\clc_d)$.
Also, for this it is enough to show \eqref{eq:lowvn1} and \eqref{eq:uppvn1}. The inequality in \eqref{eq:lowvn1} was shown in Section \ref{sec:pfoflowvn} and the proof of the inequality in \eqref{eq:uppvn1} was provided in Section \ref{sec:pfofuppvn}. Combining these we have Lemma \ref{lem:keycgce}. \hfill \qed

\section{Proof of Theorem \ref{thm:main}.}
\label{sec:pfmainth}
In order to prove the theorem it suffices to show \eqref{eq:levset} and  \eqref{eq:mainvar}. Proof of  \eqref{eq:mainvar} is given in Section \ref{sec:pfmain} while the proof of  \eqref{eq:levset}  is provided in Section \ref{sec:cptlev}.

\subsection{Proof of \eqref{eq:mainvar}.}
\label{sec:pfmain}
Let $\{\veps_n\}_{n \in \NN}$ be a sequence of positive reals such that $\veps_n \to 0$ as $n\to \infty$.
To show \eqref{eq:mainvar} it suffices to show that for every $G \in C_b(\RR)$
\begin{equation}
	\label{eq:ldpupp}
    \liminf_{n\to \infty}-\veps_n^2 \log \EE_{\PP^{\veps_n}}\left[ \exp\left\{ - \veps_n^{-2} G(V^{\veps_n}[\phi])\right\}\right]
   	\ge \inf_{z\in \RR}\{ G(z) + I^{\phi}(z)\},
\end{equation}
\begin{equation}
	\label{eq:ldplow}
    \limsup_{n\to \infty}-\veps_n^2 \log \EE_{\PP^{\veps_n}}\left[ \exp\left\{ - \veps_n^{-2} G(V^{\veps_n}[\phi])\right\}\right]
   	\le \inf_{z\in \RR}\{ G(z) + I^{\phi}(z)\}.
\end{equation}
We begin with the proof of \eqref{eq:ldpupp}.
Fix $\delta \in (0,1)$ and using \eqref{eq:repboudup} choose $(\tilde u_n, \tilde v_n) \in \cla^k\times \cla^m$ such that
\begin{equation}\label{eq:delopt}
	\begin{aligned}
	&-\veps_n^2 \log \EE_{\PP^{\veps_n}}\left[ \exp\left\{ - \veps_n^{-2} G(V^{\veps_n}[\phi])\right\}\right]\\
	&\ge \EE_{\PP^{\veps_n}}\left[G(\bar V^{\veps_n, \tilde u_n,\tilde v_n}[\phi]) + \frac{1}{2} \int_0^T (\|\tilde u_n(s)\|^2 + \|\tilde v_n(s)\|^2) ds \right] -\delta.
	\end{aligned}
\end{equation}
Note that
\begin{equation}
	\label{eq:unifmombd}
	\sup_{n\in \NN} \EE_{\PP^{\veps_n}}\left[ \int_0^T (\|\tilde u_n(s)\|^2 + \|\tilde v_n(s)\|^2) ds \right]
	\le 2(2 \|G\|_{\infty}+1) \doteq c_G.
\end{equation}
We now use a standard localization argument (cf. \cite[Theorem 3.17]{budhiraja2019analysis}). For $M\in (0,\infty)$
let
$$\tau^n_M \doteq  \inf\{t \ge 0: \int_0^t (\|\tilde u_n(s)\|^2 + \|\tilde v_n(s)\|^2) ds \ge M\}$$
and define
$$\tilde u_{n,M}(s) \doteq \tilde u_n(s)1_{[0, \tau^n_M]}(s), \;\; \tilde v_{n,M}(s) \doteq \tilde v_n(s)1_{[0, \tau^n_M]}(s), \; s \in [0,T].$$
Denoting the expectation on the right side of  \eqref{eq:delopt} by $R^n$ and denoting the corresponding expectation, with 
$(\tilde u_n, \tilde v_n)$ replaced by $(\tilde u_{n,M}, \tilde v_{n,M})$, as $R^{n,M}$ we see that
\begin{align*}
R^n - R^{n,M} &\ge - \|G\|_{\infty} \PP^{\veps_n}(\tau^n_M \le T) \\
&= - \|G\|_{\infty} \PP^{\veps_n}(\int_0^T (\|\tilde u_n(s)\|^2 + \|\tilde v_n(s)\|^2) ds \ge M)  \ge  - \|G\|_{\infty} \frac{c_G}{M},
\end{align*}
where the last inequality uses \eqref{eq:unifmombd}.
Now choose $M$ such that $\|G\|_{\infty} \frac{c_G}{M} \le \delta$ and denote $\tilde u_{n,M} = u_n$, $\tilde v_{n,M} = v_n$. Then
\begin{equation}\label{eq:deloptrev}
	\begin{aligned}
	&-\veps_n^2 \log \EE_{\PP^{\veps_n}}\left[ \exp\left\{ - \veps_n^{-2} G(V^{\veps_n}[\phi])\right\}\right]\\
	&\ge \EE_{\PP^{\veps_n}}\left[G(\bar V^{\veps_n, u_n, v_n}[\phi]) + \frac{1}{2} \int_0^T (\|u_n(s)\|^2 + \| v_n(s)\|^2) ds \right] -2\delta.
	\end{aligned}
\end{equation}
Note that $\{(u_n, v_n)\}$ is a sequence of  $S_M$ valued random variable and since $S_M$ is weakly compact, every subsequence of $\{(u_n, v_n)\}$ has a weakly convergent subsubsequence. It suffices to show \eqref{eq:ldpupp} along such a  subsubsequence which we denote once more as $\{n\}$. 
Denoting the limit as $(u,v)$, given on some probability space $(\Om^0, \clf^0, \PP^0)$, we have from Lemma \ref{lem:keycgce}
that, as $n\to \infty$,
  $\bar V^{\veps_n, u_n,v_n}[\phi] \to V_0^{u, v}[\phi]$, in distribution.
  Using the fact that $G \in C_b(\RR)$ and Fatou's lemma, we now have
  \begin{align*}
 & \liminf_{n\to \infty}	\EE_{\PP^{\veps_n}}\left[G(\bar V^{\veps_n, u_n, v_n}[\phi]) + \frac{1}{2} \int_0^T (\|u_n(s)\|^2 + \| v_n(s)\|^2) ds\right]\\
 &\ge \EE_{\PP^{0}}\left[G( V^{ u, v}_0[\phi]) + \frac{1}{2} \int_0^T (\|u(s)\|^2 + \| v(s)\|^2) ds\right]\\
  &\ge \EE_{\PP^{0}}\left[G( V^{ u, v}_0[\phi]) + I^{\phi}(V^{ u, v}_0[\phi])\right]
  \ge \inf_{z\in \RR}[G(z) + I^{\phi}(z)],
  \end{align*}
  where the second inequality uses the fact that, by definition $(u,v) \in \cls(V^{ u, v}_0[\phi])$ a.s. 
  Combining the above display with \eqref{eq:deloptrev} and recalling that $\delta>0$ is arbitrary, we have \eqref{eq:ldpupp}.
  
  We now give the proof of \eqref{eq:ldplow}.
  Fix $\delta \in (0,1)$ and let $z^* \in \RR$ be such that
  \begin{equation}\label{eq:eq153}
  	 G(z^*)  + I^{\phi}(z^*) \le \inf_{z\in \RR}[G(z) + I^{\phi}(z)]+\delta.
  \end{equation}
  Now choose $(\varphi, \psi) \in \cls(z^*) $ such that
  \begin{equation}\label{eq:eq152}
  	\frac{1}{2} \int_0^T \|\varphi(t)\|^2 dt + \frac{1}{2} \int_0^T \|\psi(t)\|^2 dt \le I^{\phi}(z^*) + \delta.
  \end{equation}
  Since $(\varphi, \psi)  \in \cla_k\times \cla_m$ (as they are non-random and square-integrable), we have from \eqref{eq:repboudup} that, for every $n\in \NN$,
  \begin{equation}\label{eq:eq149}
	\begin{aligned}
		&-\veps_n^2 \log \EE_{\PP^{\veps_n}}\left[ \exp\left\{ - \veps_n^{-2} G(V^{\veps_n}[\phi])\right\}\right]\\
	&\le\EE_{\PP^{\veps_n}}\left[G(\bar V^{\veps_n, \varphi,\psi}[\phi]) + \frac{1}{2} \int_0^T (\|\varphi(s)\|^2 + \|\psi(s)\|^2) ds \right].
	\end{aligned}
  \end{equation}
  Also, from Lemma \ref{lem:keycgce}, as $n\to \infty$,
  $\bar V^{\veps_n, \varphi,\psi}[\phi] \to V_0^{\phi, \psi}[\phi]$, in distribution.
  Since $(\varphi, \psi) \in \cls(z^*) $, \eqref{eq:clszch} holds with $z$ replaced with $z^*$ and so 
  $V_0^{\phi, \psi}[\phi]=z^*$.
  Thus sending $n\to \infty$ in \eqref{eq:eq149}, we have
  \begin{align*}
&  \limsup_{n\to \infty}-\veps_n^2 \log \EE_{\PP^{\veps_n}}\left[ \exp\left\{ - \veps_n^{-2} G(V^{\veps_n}[\phi])\right\}\right]\\
  	&\le G(z^*) + \frac{1}{2} \int_0^T (\|\varphi(s)\|^2 + \|\psi(s)\|^2) ds \le G(z^*) + I^{\phi}(z^*) + \delta \le \inf_{z\in \RR}[G(z) + I^{\phi}(z)]+2\delta,
\end{align*}
where the second inequality uses \eqref{eq:eq152} while the third uses \eqref{eq:eq153}.
Since $\delta>0$ is arbitrary, we have \eqref{eq:ldplow}, and, together with \eqref{eq:ldpupp}, completes the proof of \eqref{eq:mainvar}. \hfill \qed

\subsection{Proof of \eqref{eq:levset}. }
\label{sec:cptlev}
Fix $\phi \in C_b(\clc_d)$ and $M \in (0,\infty)$. Consider the set
$\{z\in \RR: I^{\phi}(z) \le M\} \doteq E_M$ and let $\{z_n\}_{n\in \NN}$ be a sequence in this set.
Since  for each $n \in \NN$, $I^{\phi}(z_n) \le M$, we can find $(\varphi_n, \psi_n) \in \cls(z_n) \subset \cll^2_k \times \cll^2_m$ such that
\begin{equation}\label{eq:eq320}
\frac{1}{2} \int_0^T (\|\varphi_n(s)\|^2 + \|\psi_n(s)\|^2) ds \le M+ \frac{1}{n}.
\end{equation}
Since $(\varphi_n, \psi_n) \in \cls(z_n)$,
\begin{equation}\label{eq:forzn}
z_n = V_0^{\varphi_n, \psi_n}[\phi] = \inf_{\eta \in \clc_d} \left[ H(\eta, \xi_0^{\varphi_n}, \psi_n) + \phi(\eta) + J(\eta)\right]
	- \inf_{\eta \in \clc_d} \left[ H(\eta, \xi_0^{\varphi_n}, \psi_n)  + J(\eta)\right].
\end{equation}
Note that, we can write
\begin{align*}
	H(\eta, \xi_0^{\varphi_n}, \psi_n) &= \frac{1}{2} \int_0^T \|h(\eta(s)) - h(\xi_0^{\varphi_n}(s)) - \psi_n(s)\|^2 ds \\
	&= \frac{1}{2} \int_0^T \|h(\eta(s)) - h(\xi_0^{\varphi_n}(s)) \|^2 ds
	- \int_0^T [h(\eta(s)) - h(\xi_0^{\varphi_n}(s))]\cdot \psi_n(s) ds\\
	&\quad + \frac{1}{2} \int_0^T \|\psi_n(s)\|^2 ds \\
	&= \tilde H(\eta, \xi_0^{\varphi_n}, \psi_n) + \frac{1}{2} \int_0^T \|\psi_n(s)\|^2 ds,
\end{align*}
where for $\eta, \tilde \eta \in \clc_d$ and $\psi \in \cll^2_m$
$$\tilde H(\eta, \tilde \eta, \psi) \doteq \frac{1}{2} \int_0^T \|h(\eta(s)) - h(\tilde \eta(s)) \|^2 ds
	- \int_0^T [h(\eta(s)) - h(\tilde \eta(s))]\cdot \psi(s) ds.$$
From \eqref{eq:forzn} and relation between $H$ and $\tilde H$ it follows that
\begin{equation}\label{eq:forznn}
z_n\doteq  \inf_{\eta \in \clc_d} \left[ \tilde H(\eta, \xi_0^{\varphi_n}, \psi_n) + \phi(\eta) + J(\eta)\right]
	- \inf_{\eta \in \clc_d} \left[ \tilde H(\eta, \xi_0^{\varphi_n}, \psi_n)  + J(\eta)\right].
\end{equation}

Also, from \eqref{eq:eq320} it follows that $\{(\varphi_n, \psi_n)\}_{n\in \NN} \subset S_{2(M+1)}$. Since $S_{2(M+1)}$ is compact, we can find a subsequence along which $(\varphi_n, \psi_n)$ converges to some $(\varphi, \psi) \in S_{2(M+1)}$. In fact, from
\eqref{eq:eq320} and lower semicontinuity it follows that $(\varphi, \psi) \in S_{2M}$. 
Define
\begin{equation}\label{eq:forzstar}
	\begin{aligned}
z^* \doteq V_0^{\varphi, \psi}[\phi] &= \inf_{\eta \in \clc_d} \left[ H(\eta, \xi_0^{\varphi}, \psi) + \phi(\eta) + J(\eta)\right]
	- \inf_{\eta \in \clc_d} \left[ H(\eta, \xi_0^{\varphi}, \psi)  + J(\eta)\right]\\
	&= \inf_{\eta \in \clc_d} \left[ \tilde H(\eta, \xi_0^{\varphi}, \psi) + \phi(\eta) + J(\eta)\right]
	- \inf_{\eta \in \clc_d} \left[ \tilde H(\eta, \xi_0^{\varphi}, \psi)  + J(\eta)\right].
\end{aligned}
\end{equation}
In order to complete the proof of \eqref{eq:levset} it suffices to show that
\begin{equation}\label{eq:zntozstar}
	\mbox{ as } n \to \infty, \;\; z_n \to z^*.
\end{equation}

We first argue that in the infimum appearing in (the second line of )\eqref{eq:forzstar} and \eqref{eq:forznn}, $\{\eta \in \clc_d\}$ can be replaced by 
$\{\eta \in K\}$ for some fixed compact set $K$.
To see this, note that, with $\xi^*$ as in \eqref{eq:intro1},
\begin{align*}
	\inf_{\eta \in \clc_d} \left[ \tilde H(\eta, \xi_0^{\varphi_n}, \psi_n) + \phi(\eta) + J(\eta)\right]&\le
	\tilde H(\xi^*, \xi_0^{\varphi_n}, \psi_n) + \|\phi\|_{\infty}+ J(\xi^*).
\end{align*}
Also, note that $J(\xi^*) =0$ and
\begin{align*}
	\tilde H(\xi^*, \xi_0^{\varphi_n}, \psi_n) &= \frac{1}{2} \int_0^T \|h(\xi^*(s)) - h(\xi_0^{\varphi_n}(s)) \|^2 ds
-  \int_0^T [h(\xi^*(s)) - h(\xi_0^{\varphi_n}(s))]\cdot \psi_n(s) ds\\
&\le \int_0^T \|h(\xi^*(s)) - h(\xi_0^{\varphi_n}(s)) \|^2 + \frac{1}{2}\int_0^T \|\psi_n(s)\|^2 ds\\
	&\le 2T\|h(\xi^*(\cdot))\|_*^2 +  2 \int_0^T \|h(\xi_0^{\varphi_n}(s))\|^2 ds + \frac{1}{2} \int_0^T \|\psi_n(s)\|^2 ds\\
	&\le 2T\|h(\xi^*(\cdot))\|_*^2 +  \kappa_1   (M+1) \doteq \kappa_2,
\end{align*}
where $\kappa_1 \in (0,\infty)$  depends only on $x_0, T$ and the linear growth coefficients of $h, b, \sigma$.
Thus, taking $\kappa_3 \doteq \kappa_2 + \|\phi\|_{\infty}+ 1$, we see that
the first infimum in 
 \eqref{eq:forznn} can be replaced by infimum over the set
$$K_0^n \doteq \{\eta \in \clc_d: \tilde H(\eta, \xi_0^{\varphi_n}, \psi_n) + \phi(\eta) + J(\eta) \le \kappa_3\}.$$
Using the relation $a\cdot b \ge -\frac{1}{4}\|a\|^2 - \|b\|^2$,
\begin{align*}
	\tilde H(\eta, \xi_0^{\varphi_n}, \psi_n) &\ge \frac{1}{2} \int_0^T \|h(\eta(s)) - h(\xi_0^{\varphi_n}(s)) \|^2 \\
	&\quad - \frac{1}{4} \int_0^T \|h(\eta(s)) - h(\xi_0^{\varphi_n}(s)) \|^2
	- \int_0^T \|\psi_n(s)\|^2 ds \ge -2M.
\end{align*}
Thus,  with $\kappa_4 \doteq  \kappa_3 + \|\phi\|_{\infty} + 1+2M$,
$K_0^n$ is contained in the compact set
$$K \doteq \{\eta \in \clc_d:  J(\eta) \le \kappa_4\}.$$
Thus the first infimum in 
 \eqref{eq:forznn} can be replaced by infimum over the set $K$. Similarly, the second infimum in \eqref{eq:forznn} and both infima in (second line of)
 \eqref{eq:forzstar} can be replaced by infima over the same compact set $K$.
 Note that if $B_n, B$ are maps from $K \to \RR$ such that $B_n \to B$ uniformly on compacts, then
 $$\inf_{\eta\in K} [B_n(\eta)+ J(\eta)] \to \inf_{\eta\in K} [B(\eta)+ J(\eta)].$$
 Thus, to complete the proof of \eqref{eq:zntozstar} it suffices to show that, 
 \begin{equation}\label{eq:eq501}
 	\mbox{ as } n \to \infty,\;\;  \tilde H(\eta, \xi_0^{\varphi_n}, \psi_n) \to \tilde H(\eta, \xi_0^{\varphi}, \psi), \; \mbox{ uniformly for } \eta \in K.
 \end{equation}
 For this note that from Assumption \ref{assu:main} and the convergence of $\varphi_n \to \varphi$ it follows that,  $\xi_0^{\varphi_n} \to \xi_0^{\varphi}$ in $\clc_d$ as $n\to \infty$.
 Also, since $K$ is compact, $\sup_{\eta \in K} \|\eta\|_* <\infty$. Combining these observations with the continuity and linear growth of $h$ we have that, as $n\to \infty$,
 \begin{equation} \label{eq:eq458}\frac{1}{2} \int_0^T \|h(\eta(s)) - h(\xi_0^{\varphi_n}(s)) \|^2 ds  \to \frac{1}{2} \int_0^T \|h(\eta(s)) - h(\xi_0^{\varphi}(s)) \|^2 ds\end{equation}
 uniformly for $\eta \in K$.
 Also, writing
 $$\int_0^T  h(\xi_0^{\varphi_n}(s))\cdot \psi_n(s) ds = 
 \int_0^T  [h(\xi_0^{\varphi_n}(s)) - h(\xi_0^{\varphi}(s))]\cdot \psi_n(s) ds +
 \int_0^T  h(\xi_0^{\varphi}(s))\cdot \psi_n(s) ds$$
 and using the convergence $(\xi_0^{\varphi_n}, \psi_n) \to (\xi_0^{\varphi}, \psi)$, the bound in \eqref{eq:eq320}, and the Lipschitz property of $h$, we have that, as $n\to \infty$
 \begin{equation}\label{eq:eq459}\int_0^T  h(\xi_0^{\varphi_n}(s))\cdot \psi_n(s) ds \to \int_0^T  h(\xi_0^{\varphi}(s))\cdot \psi(s) ds.\end{equation}
 Finally we claim that, as $n\to \infty$, 
 \begin{equation}\label{eq:eq500}
	 \int_0^T h(\eta(s)) \cdot \psi^n(s) ds \to \int_0^T h(\eta(s)) \cdot \psi(s) ds,\end{equation}
 uniformly for $\eta \in K$.
 Indeed, to show the claim, it suffices to show that if $\eta^n \to \eta$ in $K$ then
 \begin{equation}\label{eq:eq455}
	 \int_0^T h(\eta^n(s)) \cdot \psi^n(s) ds \to \int_0^T h(\eta(s)) \cdot \psi(s) ds.
	 \end{equation}
 Write the right hand side as
 $$\int_0^T h(\eta^n(s)) \cdot \psi^n(s) ds = \int_0^T (h(\eta^n(s))-h(\eta(s))) \cdot  \psi^n(s) ds
 +  \int_0^T h(\eta(s)) \cdot \psi^n(s) ds.$$
 The convergence in \eqref{eq:eq455} is now immediate from the above display on using, the Lipschitz  property of $h$, the bound in \eqref{eq:eq320}, and the convergence of
 $(\eta^n, \psi^n)$ to $(\eta, \psi)$, which proves the claim.
 Combining the  convergence properties in \eqref{eq:eq458}, \eqref{eq:eq459}, and \eqref{eq:eq500}, we now have the statement in \eqref{eq:eq501}, which, as noted previously, proves \eqref{eq:levset}. \hfill \qed

\noindent \textbf{Acknowledgement:} 
The research of AB was supported in part by the NSF (DMS-1814894 and DMS-1853968). 
Part of this research was carried out when AB was visiting the International Centre for
	Theoretical Sciences-TIFR  and Imperial College, London, and he will like to thank his hosts Amit Apte and Dan Crisan, respectively,  at these institutes,  for their generous hospitality. AA and ASR acknowledge the
	support of the Department of Atomic Energy, Government of India, under
	projects no.12-R\&D-TFR-5.10-1100, and no.RTI4001.
 \bibliographystyle{plain}
\bibliography{filt-ldp3}
\end{document}